\RequirePackage[2020-02-02]{latexrelease}

\documentclass[
aip,%
amsmath,amssymb,
reprint,%
]{revtex4-1}

\usepackage{graphicx}
\usepackage{subcaption}
\usepackage{dcolumn}
\usepackage{bm}
\usepackage[normalem]{ulem}
\usepackage[utf8]{inputenc}
\usepackage[T1]{fontenc}
\usepackage{amsthm}
\usepackage{mathptmx}
\usepackage{etoolbox}
\usepackage[dvipsnames,table,xcdraw]{xcolor}
\usepackage[ruled]{algorithm2e}
\usepackage{dsfont}
\usepackage{amsmath}
\theoremstyle{plain}
\newtheorem{theorem}{Theorem}[section]

\newtheorem{proposition}[theorem]{Proposition}

\newcommand{\Ob}{\boldsymbol{0}}
\newcommand{\xb}{\boldsymbol{x}}
\newcommand{\yb}{\boldsymbol{y}}
\newcommand{\zb}{\boldsymbol{z}}

\newcommand{\bb}{\boldsymbol{b}}
\newcommand{\db}{\boldsymbol{d}}

\newcommand{\Ab}{\boldsymbol{A}}
\newcommand{\Fb}{\boldsymbol{F}}
\newcommand{\Gb}{\boldsymbol{G}}
\newcommand{\ub}{\boldsymbol{u}}
\newcommand{\Rb}{\boldsymbol{R}}
\newcommand{\Ib}{\boldsymbol{I}}
\newcommand{\eb}{\boldsymbol{e}}
\newcommand{\thetab}{\boldsymbol{\theta}}

\newcommand{\Kb}{\boldsymbol{K}}
\newcommand{\Lc}{\mathcal{L}}

\newcommand{\appropto}{\mathrel{\vcenter{
  \offinterlineskip\halign{\hfil$##$\cr
    \propto\cr\noalign{\kern2pt}\sim\cr\noalign{\kern-2pt}}}}}
    
\makeatletter
\def\@email#1#2{%
 \endgroup
 \patchcmd{\titleblock@produce}
  {\frontmatter@RRAPformat}
  {\frontmatter@RRAPformat{\produce@RRAP{*#1\href{mailto:#2}{#2}}}\frontmatter@RRAPformat}
  {}{}
}%
\makeatother

\begin{document}

\preprint{AIP/123-QED}

\title{The fast committor machine: Interpretable prediction with kernels}

\author{David Aristoff}
\affiliation{Mathematics, Colorado State University, Fort Collins, CO, USA}
\thanks{Author to whom correspondence should be addressed: \url{aristoff@colostate.edu}}
\author{Mats Johnson}
\affiliation{Mathematics, Colorado State University, Fort Collins, CO, USA}
\author{Gideon Simpson}
\affiliation{Mathematics, Drexel University, Philadelphia, PA, USA}%
\author{Robert J. Webber}
\affiliation{Mathematics, University of California San Diego, La Jolla, CA, USA}


\begin{abstract}
In the study of stochastic systems, 
the committor function describes the probability that a system 
starting from an initial configuration $x$ will reach a set $B$ before a set $A$.
This paper introduces an efficient and interpretable algorithm for approximating the committor, called the ``fast committor machine'' (FCM).
The FCM uses simulated trajectory data to build a kernel-based model of the committor.
The kernel function is constructed to emphasize low-dimensional subspaces that optimally describe the $A$ to $B$ transitions.
The coefficients in the kernel model are determined using randomized linear algebra, leading to a runtime that scales linearly in the number of data points.
In numerical experiments involving a triple-well potential and alanine dipeptide, the FCM yields higher accuracy and trains more quickly than a neural network with the same number of parameters.
The FCM is also more interpretable than the neural net.
\end{abstract}

\maketitle

\section{Introduction}

Many physical systems exhibit metastability, which is the tendency to occupy a region $A$ of phase space for a comparatively long time before a quick transition to another region $B$.
Metastability is especially common in molecular dynamics, where states $A$ and $B$ might correspond to the folded and unfolded configurations of a protein.
Metastability also arises in continuum mechanics \cite{newhall_metastability_2017}, biological systems \cite{grafke_numerical_2019}, fluids \cite{jacques-dumas_data-driven_2023}, and the Earth's climate \cite{Lucente_Herbert_Bouchet_2022}.

Scientists can gain insight into metastable systems by numerically simulating and then analyzing the transition paths from $A$ to $B$. 
The transition paths can be simulated by a variety of methods including transition path sampling \cite{bolhuis_transition_2002,rogal_multiple_2008}, forward flux sampling \cite{allen_forward_2006,escobedo_transition_2009}, the string method \cite{e_finite_2005,vanden-eijnden_revisiting_2009}, and adaptive multilevel splitting \cite{cerou_multiple_2011}.
This work assumes that a data set of trajectories has already been simulated, including one or more transitions between $A$ and $B$.
A more detailed discussion of sampling methods appears in the conclusion.

Scientists can study the transitions from $A$ to $B$ by evaluating the (forward) committor function \cite{khoo_solving_2019,Li2019,li_semigroup_2022,evans_computing_2022,jacques-dumas_data-driven_2023,lai_point_2018,lucente_coupling_2021,Lucente_Herbert_Bouchet_2022,berezhkovskii_committors_2019}, which measures the probability that the system starting at state $\xb$ will reach $B$ before $A$, thus making a transition. 
In symbols, the committor is given by
\begin{equation*}
    q^*(\xb) = \mathbb{P}_{\xb}(T_B < T_A),
\end{equation*}
where $T_A$ and $T_B$ are the first hitting times for $A$ and $B$.

To numerically approximate the committor function, the trajectory data needs to be combined with an appropriate 
numerical
method.
Beyond the direct counting of $A$ and $B$ transitions,
there are now various machine learning methods available, including neural networks \cite{khoo_solving_2019,Li2019,li_semigroup_2022}, diffusion maps \cite{evans_computing_2022}, 
Markov state models \cite{pande2010everything,chodera2014markov,husic2018markov},
and other tools \cite{lai_point_2018}.
As shortcomings, these methods can be slow to apply to large data sets (diffusion maps) or difficult to interpret (neural nets).

This work describes a new method called the fast committor machine (FCM) that is both efficient and interpretable.
The FCM approximates the committor $q^*$ using a linear combination of kernel functions, as follows:
\begin{itemize}
    \item The kernel function is defined in an interpretable way, as the composition of an exponential kernel with an adaptively chosen linear map.
    The linear map is generated by the recursive feature machine (RFM~\cite{radhakrishnan2022feature}).
    The map emphasizes the low-dimensional subspaces that make the greatest contribution to the gradient of the committor and suppresses other subspaces.

    \item The coefficients in the kernel approximation are optimized using randomly pivoted Cholesky (RPC~\cite{chen_randomly_2023}).
    RPC is a randomized linear algebra approach that processes $N$ data points using just $\mathcal{O}(N)$ floating point operations and $\mathcal{O}(N)$ storage.
\end{itemize}
The RFM and RPC were introduced in 2023--2024 and appear to be new in the chemical physics literature. 
Adapting these techniques to the committor problem requires going beyond the original papers \cite{radhakrishnan2022feature,chen_randomly_2023} and introducing a new functional form based on a variational formulation of the committor.

In summary, the FCM is a method for approximating the committor that is based on an adaptively chosen linear transformation of the phase space.
In numerical experiments,
the method is more accurate and trains more quickly than a neural network with the same number of parameters.
As another advantage, the FCM is interpretable, revealing low-dimensional linear subspaces that optimally describe the transition pathways from $A$ to $B$.

\subsection{Relationship to past work}

Kernel methods are frequently used in chemical machine learning \cite{ferguson2011nonlinear,rohrdanz2011determination,griffiths2023gauche}, especially for calculating force fields \cite{burn2020creating,deringer2021gaussian,burn2022ichor,tom2023calibration}.
Yet the efficiency of the kernel methods depends on finding a suitable distance function between molecular configurations \cite{musil2021physics}.

Two distances are currently being advocated in the chemical physics literature.
Diffusion maps \cite{coifman_diffusion_2008} provide a sophisticated distance that adapts to the nonlinear manifold structure in the input data, and diffusion maps have recently been used to approximate the committor \cite{evans_computing_2022}.
Alternatively, the variational approach to conformational dynamics \cite{noe2013variational} identifies a coordinate transformation that amplifies slowly decorrelating dynamical modes, and VAC is often used to help visualize and construct committor estimates \cite{chen2023discovering,chen2023chasing,schutte2023overcoming}.

As a limitation, diffusion maps and VAC do not adapt to the sets $A$ and $B$ used in the committor definition.
Greater accuracy could be potentially obtained by modifying the distance based on $A$ and $B$.
However, this modification requires a fully adaptive learning strategy that redefines the kernel distance for each committor problem.

The kernel machine learning literature has recently developed a flexible approach to distance construction.
Many authors apply a \emph{linear} transformation to their data before calculating a standard exponential or square exponential kernel \cite{kamada2006support,campsvalls2007hyperspectral}.
The linear map can be adapted to the input data and also to the specific approximation task \cite{wu_learning_2010,trivedi2014consistent}.
A 2024 paper~\cite{radhakrishnan2022feature} introduced the \emph{recursive feature machine} (RFM), which constructs the linear feature map based on the estimated gradients of the target function.
To justify the RFM, the authors compared the linear map with the first layer of a neural net, and they showed empirical and theoretical similarities \cite{radhakrishnan2022feature,beaglehole2023mechanism,beaglehole2024average}.
In experiments, they applied the RFM to RGB images with as many as $96 \times 96$ pixels, hence $96 \times 96 \times 3 = 27,648$ dimensions.
They found that the RFM suppresses irrelevant directions and emphasizes important directions, which is especially important in high dimensions.
For further justification of the RFM's effectiveness, see Appendix~\ref{sec:scaling}.

The RFM provides the template for a kernel-based approximation of the committor.
Yet optimizing the coefficients in the kernel approximation remains computationally demanding.
A naive optimization would require $\mathcal{O}(N^3)$ operations and $\mathcal{O}(N^2)$ storage, which is prohibitively expensive for $N \geq 10^5$ data points.
The high computational cost of kernel methods has been described as a fundamental limitation in the past \cite{unke2021machine}.
Nonetheless, modern strategies in randomized numerical linear algebra are leading to dramatic speed-ups \cite{williams2000using,rudi2017falkon,chen_randomly_2023,díaz2023robust}.

This work applies randomly pivoted Cholesky (RPC~\cite{chen_randomly_2023}) to speed up the kernel optimization and extend the numerical experiments to $N = 10^6$ data points.
To achieve these speed-ups, RPC generates a randomized rank-$r$ approximation of the kernel matrix, where $r$ is a parameter chosen by the user.
The coefficients can be optimized in just $\mathcal{O}(N r^2)$ operations and $\mathcal{O}(N r)$ storage.
For the experiments in Sec.~\ref{sec:experiments}, a constant value of $r = 1000$ yields nearly converged committor results. 

\subsection{Outline for the paper}
The rest of the paper is organized as follows.
Section \ref{sec:context} provides context for the committor approximation task,
Section \ref{sec:new_method} describes the new FCM method for calculating the committor, Section \ref{sec:experiments} presents numerical experiments, and Section \ref{sec:conclusion} offers concluding remarks.
Technical derivations are in
Appendices \ref{sec:alternative} and \ref{sec:scaling}.

\subsection{Assumptions and notation}

The FCM can be applied to any time-reversible Markovian system $\xb(t) \in \mathbb{R}^d$ which has a discrete time step $\tau$ and equilibrium density $\mu$.
See Table~\ref{tab:symbol-definitions} for a list of symbols and definitions.

\begin{table}[t]
    \caption{Definitions of symbols used in this work.}
    \label{tab:symbol-definitions}
    \centering
    \begin{tabular}{c | @{\hspace{1em}}l@{}}
        Symbol & Definition \\ [0.5ex] \hline
        $\xb$, $\xb'$ & system states \\
        $A,B$ & initial and target sets \\
        $\Omega$ &   complement of $A \cup B$ \\
        $\xb(t)$ & underlying stochastic process \\
        $\tau$ &  lag time \\
        $q^*(\xb)$ & exact committor  \\
        $\thetab$ & linear coefficients in kernel approximation \\
        $q_{\thetab}(\xb)$ & estimated committor  \\
        $k_{\bm{M}}(\xb,\xb')$ & kernel function \\
        $\bm{M}$ & scaling matrix \\
        $\Kb$ & kernel matrix \\
        $\Ib$ & \textup{identity matrix} \\
        $(\xb_n,\yb_n)$ & time lagged pairs, $\xb(0)=\xb_n$, $\xb(\tau)=\yb_n$ \\
        $N$ & number of training samples \\
        $\mu$ & equilibrium density \\
        $\rho$ & sampling density \\
        $w_n = \mu(\xb_n)/\rho(\xb_n)$ & weight or likelihood of $\xb_n$ \\
        $1_A, 1_B, 1_\Omega$ & characteristic functions of $A,B,\Omega$ \\
        $\varepsilon$ & bandwidth parameter \\
        $\gamma$ & regularization parameter
    \end{tabular}
\end{table}

\section{Context} \label{sec:context}

The FCM is motivated by a variational principle which states that the committor $q^*(\xb) = \mathbb{P}_{\xb}(T_B < T_A)$ is the unique minimizer of
\begin{equation}
\label{eq:dirichlet}
    \Lc(q) = \mathbb{E}_{\mu} \bigl|q(\bm{x}(0)) - q(\bm{x}(\tau))\bigr|^2,
\end{equation}
among square-integrable functions satisfying $q = 0$ on $A$ and $q = 1$ on $B$.
The expectation $\mathbb{E}_{\mu}$ is an average over all trajectories $\bm{x}(t)$ that are started from the equilibrium density $\bm{x}(0) \sim \mu$ and run forward to an end point $\bm{x}(\tau)$.
The functional \eqref{eq:dirichlet} is  known as the discrete-in-time Dirichlet form, and it appears frequently in Markov chain analysis \cite{doyle1984random,hersch1969brownian,schmuland1999dirichlet,fukushima_dirichlet_2011}.
Here, the Dirichlet form serves as a cost function for committor estimation in machine learning \cite{chen2023discovering,chen2023chasing}.

Frequently,
the Dirichlet form needs to be approximated from data.
Assume the data consists of $N$ pairs $(\xb_n, \yb_n)_{n = 1, 2, \ldots, N}$, where the start point $\xb_n$ is randomly sampled from a density $\rho$ and the end point $\yb_n$ is the result of running the model forward for $\tau$ time units.
Then, a data-driven approximation for $\mathcal{L}(q)$ is the weighted average
\begin{equation} \label{eq:loss}
    L(q)
    = \frac{1}{N} \sum_{n=1}^N w_n \bigl|q(\xb_n)-q(\yb_n)\bigr|^2,
\end{equation}
where the weights are the likelihood ratios between the sampling density $\rho$ and the target density $\mu$:
\begin{equation} \label{eq:weights}
    w_n = \frac{\mu(\bm{x}_n)}{\rho(\bm{x}_n)}.
\end{equation}
The weighted average \eqref{eq:loss} is mathematically justified.
If the data pairs $(\bm{x}_n, \bm{y}_n)_{n = 1, 2, \ldots}$ are ergodic and $\mu$ is \emph{absolutely continuous} with respect to $\rho$:
\begin{equation*}
    \rho(\bm{x}) > 0 \quad \text{whenever} \quad \mu(\bm{x}) > 0,
\end{equation*}
then the estimator \eqref{eq:loss} is unbiased
\begin{equation*}
    \mathbb{E}[L(q)] = \mathcal{L}(q),
\end{equation*}
and the law of large numbers guarantees $L(q) \rightarrow \mathcal{L}(q)$ as $N \rightarrow \infty$.
Nonetheless, consistent with the law of large numbers, the accuracy of the weighted average may depend on generating a large quantity of data.

When the goal is variational minimization of $L(q)$, 
there is an option to use alternative weights
\begin{equation}
\label{eq:alternative}
    w_n' = c \frac{\mu(\bm{x}_n)}{\rho(\bm{x}_n)},
\end{equation}
where the multiplicative constant $c > 0$ can be any positive number based on convenience.
This makes the variational approach applicable even when the likelihood ratios $\mu(\bm{x}_n) / \rho(\bm{x}_n)$ are only known up to a constant multiplicative factor.

\section{New FCM method for calculating the committor} \label{sec:new_method}

This section describes the new method for calculating the committor function $q^*$, called the ``fast committor machine'' (FCM).

\subsection{Form of the committor approximation}

Recall that $(\xb_n, \yb_n)_{n = 1, 2, \ldots, N}$ is a sequence of simulated data pairs.
The FCM is based on a new data-driven committor approximation of the form
\begin{equation} \label{eq:ansatz}
\begin{aligned}
    q_{\bm{\theta}}(\bm{x}) &= 0,
    &\bm{x} \in A, \\
    q_{\bm{\theta}}(\bm{x}) &= 1,
    &\bm{x} \in B, \\
    q_{\bm{\theta}}(\bm{x}) 
    &= 
    \sum_{n=1}^N 
    \theta_n \bigl[ k_{\bm{M}}(\bm{x}_n, \bm{x}) - k_{\bm{M}}(\bm{y}_n, \bm{x})],
    &\bm{x} \in \Omega.
\end{aligned}
\end{equation}
In this definition, $\Omega = (A \cup B)^c$ is the region of unknown committor values, $\bm{\theta} \in \mathbb{R}^N$ is a vector of coefficients, and
\begin{equation}
\label{e:kMkernel}
    k_{\bm{M}}(\bm{x}, \bm{x}') 
    = \begin{cases}
        \exp\bigl(-\tfrac{1}{\varepsilon} \lVert \bm{M}^{1/2} (\bm{x} - \bm{x}') \rVert \bigr), & \bm{x}, \bm{x}' \in \Omega \\
        0, & \text{otherwise}
    \end{cases}
\end{equation}
is a kernel function with a bandwidth parameter $\varepsilon > 0$ and a positive definite scaling matrix $\bm{M} \in \mathbb{R}^{d \times d}$.

The approximation $q_{\bm{\theta}}$ is more parameter-efficient and empirically accurate than other kernel-based parametrizations considered in Appendix \ref{sec:alternative}.
By construction, $q_{\bm{\theta}}$ satisfies the appropriate boundary conditions.
Also, $q_{\bm{\theta}}$ changes flexibly in the interior region $\Omega$.
Note that the kernel $k_{\bm{M}}(\bm{x}, \bm{y})$ is not differentiable at $\bm{x} = \bm{y}$, but this paper uses $\nabla k_{\bm{M}}(\bm{x}, \bm{x}) = \bm{0}$ as the pseudogradient.

The rest of this section describes the approach for optimizing the scaling matrix $\bm{M}$, the coefficient vector $\bm{\theta}$, and the bandwidth $\varepsilon$ in the FCM.

\subsection{Optimization of the scaling matrix $\bm{M}$}

The FCM incorporates a scaling matrix $\bm{M} \in \mathbb{R}^{d \times d}$ that transforms the state variable $\bm{x} \in \mathbb{R}^d$ according to $\bm{x} \mapsto \bm{M}^{1/2} \bm{x}$.
The optimization of this scaling matrix is key to the performance of the FCM.

Since the early 2000s, researchers have suggested a scaling matrix chosen as the inverse covariance matrix of the input data points
\cite{kamada2006support,campsvalls2007hyperspectral}.
This choice of $\bm{M}$ transforms the data points so they become isotropic.
Isotropy can sometimes improve the predictive accuracy, but there is a better approach to optimally select $\bm{M}$ for kernel learning.

Radhakrishnan et al. recently introduced an approach for tuning the scaling matrix called the ``recursive feature machine'' (RFM) \cite{radhakrishnan2022feature}.
The RFM is a generalized kernel method that takes input/output pairs $(\bm{x}_n, b_n = f(\xb_n))_{n = 1, 2, \ldots, N}$ and iteratively learns both a regressor 
\begin{equation} \label{eq:regressor}
    f_{\bm{\theta}}(\xb) = \sum_{n=1}^N \theta_n\, k_{\bm{M}}(\xb_n, \xb)
\end{equation}
and a scaling matrix $M$.
The RFM alternates between two steps:
\begin{itemize}
    \item Update the coefficient vector $\bm{\theta}$ so that $f_{\bm{\theta}}$ minimizes the least-squares loss
    \begin{equation*}
        \thetab \gets \arg \min_{\thetab} \frac{1}{N} \sum_{n = 1}^N |f_{\bm{\theta}}(\bm{x}_n) - b_n|^2
    \end{equation*}
    using the current scaling matrix $\bm{M}$.
    \item Update the scaling matrix 
    \begin{equation} \label{eq:scaling}
        \bm{M} = \sum_{n=1}^N \nabla f_{\thetab}(\xb_n) \nabla f_{\thetab}(\xb_n)^T.
    \end{equation}
    using the current regressor $f_{\bm{\theta}}$. 
\end{itemize}
The method iterates until finding an approximate fixed point for $f_{\bm{\theta}}$ and $\bm{M}$, and typically $3-6$ iterations are enough \cite{radhakrishnan2022feature}.

A complete analysis of the RFM lies beyond the scope of the current work.
However, as an intuitive explanation, 
the linear map improves the ``fit'' between the target function and the approximation.
Specifically, the change-of-basis $\bm{z} = \bm{M}^{1/2} \bm{x}$ causes $f_{\thetab}$ to have isotropic gradients $\nabla_{\bm{z}} f_{\thetab}(\bm{z}_1), \ldots, \nabla_{\bm{z}} f_{\thetab}(\bm{z}_N)$.
Thus, it becomes relatively easy to approximate $f$ using a linear combination of isotropic kernel functions.
See Appendix \ref{sec:scaling} for a proof of the isotropy property of the RFM.

The main contributions of this work are the extension of the RFM to the committor problem, together with an efficient strategy for optimizing the coefficient vector $\bm{\theta}$.
In homage to the original RFM paper \cite{radhakrishnan2022feature}, the new method is called the ``fast committor machine'' (FCM).

Pseudocode for the FCM is provided in Algorithm~\ref{alg:main}.
As a helpful feature, the pseudocode normalizes the scaling matrix $\bm{M}$ by the trace of the covariance matrix
\begin{multline*}
    \operatorname{tr}\, \operatorname{cov}(\{\bm{M}^{1/2}\xb_1, \ldots, \bm{M}^{1/2}\xb_N\}) \\
    = \frac{1}{2N(N - 1)} \sum_{m,n = 1}^N
    \lVert \bm{M}^{1/2}(\bm{x}_m - \bm{x}_n) \rVert^2.
\end{multline*}
This makes it easier to select and interpret the bandwidth parameter $\varepsilon > 0$.
For the experiments in Sec.~\ref{sec:experiments}, setting $\varepsilon = 1$ works well as a default.

\begin{algorithm}[t]
\caption{Fast committor machine} \label{alg:main}
\KwData{Training points: $(\xb_n,\yb_n,w_n)_{n = 1, \ldots, N}$, \;
Solution vector: $\bm{b}$\;
Kernel function: $k$\;
Parameters: bandwidth $\epsilon$,
regularization $\gamma$, approximation rank $r$ (multiple of $10$)\;}

\KwResult{Scaling matrix $\bm{M}$, committor estimate $q_{\bm{\theta}}$\;} 

$\bm{M} \gets \bm{I}$\;
\qquad \textcolor{blue}{\tcp{Initialize $\bm{M}$}}

\For{$t = 1, \ldots, 5$}{

$\bm{M} \gets \bm{M} / \operatorname{tr}\, \operatorname{cov}(\{\bm{M}^{1/2}\xb_1, \ldots, \bm{M}^{1/2}\xb_N\})$\;
\qquad \textcolor{blue}{\tcp{Rescale $\bm{M}$}}
$k_{mn} \gets k_{\bm{M}}\!(\bm{x}_m, \bm{x}_n) - k_{\bm{M}}\!(\bm{x}_m, \bm{y}_n) - k_{\bm{M}}\!(\bm{y}_m, \bm{x}_n)$ \\
\qquad \qquad $+ k_{\bm{M}}\!(\bm{y}_m, \bm{x}_n)$\;
\qquad \textcolor{blue}{\tcp{Repopulate $\Kb$ with kernel $k_{\bm{M}}$}}

$k_{mn} \gets \sqrt{w_m w_n}\, k_{mn}$\;
\qquad \textcolor{blue}{\tcp{Reweight $\bm{K}$}}

$S = \{s_1, \ldots, s_r\} \gets \texttt{RPCholesky}(\bm{K}, r)$\;
\qquad \textcolor{blue}{\tcp{Apply Alg.~\ref{alg:RPCholesky}}}

$\Kb(S,S) \gets \Kb(S,S) + \varepsilon_{\rm mach} \textup{tr}(\bm{K}(S, S)) \Ib$ \;
\qquad \textcolor{blue}{\tcp{Regularize}}

$\bm{\eta} \gets \textup{Solution to }$
\qquad $[\Kb(:,S)^T \Kb(:,S) + \gamma N \Kb(S,S)] \bm{\eta} = \Kb(:,S)^T \bb$\;
\qquad \textcolor{blue}{\tcp{Optimize coefficients}}

$q_{\bm{\theta}}(\bm{x}) \gets \sum_{i=1}^r 
\eta_i \sqrt{w_{s_i}} \bigl[ k_{\bm{M}}(\bm{x}_{s_i}, \bm{x}) - k_{\bm{M}}(\bm{y}_{s_i}, \bm{x})]$\;
\qquad \textcolor{blue}{\tcp{Update $q_{\bm{\theta}}$}}

$\bm{M} \gets \sum_{n=1}^N \nabla q_{\bm{\theta}} (\xb_n) \nabla q_{\bm{\theta}} (\xb_n)^T$\;
\qquad \textcolor{blue}{\tcp{Update $\bm{M}$}}

}
\KwRet{$\bm{M}$, $q_{\bm{\theta}}$}
\end{algorithm}

\subsection{Optimization of the coefficients $\bm{\theta}_n$}

To derive an efficient procedure for optimizing the coefficient vector $\bm{\theta} \in \mathbb{R}^N$,
the first step is to rewrite the optimization as a standard least-squares problem.
Define the rescaled coefficients 
\begin{equation*}
    \overline{\theta}_n = \theta_n / \sqrt{w}_n,
    \quad n = 1, \ldots, N.
\end{equation*}
Also, introduce the kernel matrix
$\bm{K} \in \mathbb{R}^{N \times N}$
and the solution vector $\bm{b} \in \mathbb{R}^N$ with entries given by
\begin{align*}
    & \begin{aligned}
    k_{mn} 
    &= \sqrt{w_m} \sqrt{w_n} \bigl[k_{\bm{M}}(\bm{x}_m, \bm{x}_n) - k_{\bm{M}}(\bm{x}_m, \bm{y}_n) \\
    & \quad - k_{\bm{M}}(\bm{y}_m, \bm{x}_n) + k_{\bm{M}}(\bm{y}_m, \bm{y}_n)\bigr],
    \end{aligned} \\
    & b_n = \sqrt{w_n} [1_B(\bm{y}_n) - 1_B(\bm{x}_n)]
\end{align*}
Plugging the FCM into the variational principle \eqref{eq:loss}, the optimal vector $\overline{\bm{\theta}} \in \mathbb{R}^N$ is the minimizer of the regularized least-squares loss:
\begin{equation}
\label{eq:linear_algebra}
    \min_{\overline{\bm{\theta}} \in \mathbb{R}^N} \tfrac{1}{N} \lVert \bm{K} \overline{\bm{\theta}} - \bm{b} \rVert^2
    + \gamma\, \overline{\bm{\theta}}^T \bm{K} \overline{\bm{\theta}}.
\end{equation}
Notice that the loss function \eqref{eq:linear_algebra} includes a regularization term $\gamma\, \overline{\bm{\theta}}^T \bm{K} \overline{\bm{\theta}}$ with $\gamma > 0$ that shrinks the norm of the coefficients to help prevent overfitting.
Specifically, $(\overline{\bm{\theta}}^T \bm{K} \overline{\bm{\theta}})^{1/2}$ is the reproducing kernel Hilbert space (RKHS) norm associated with the committor approximation $q_{\bm{\theta}}$.
RKHS norms appear frequently in the kernel literature due to their theoretical properties and computational convenience;
see Sec.~2.3 of the paper~ \cite{kanagawa2018gaussian} for an introduction.

For large data sets with $N \geq 10^5$ data points, it is computationally convenient to use a randomized strategy for solving
\eqref{eq:linear_algebra}.
The randomly pivoted Cholesky algorithm (Algorithm \ref{alg:RPCholesky}) selects a set of ``landmark'' indices 
\begin{equation*}
    S = \{s_1, s_2, \ldots, s_r\}.
\end{equation*}
Then, the coefficient vector $\overline{\bm{\theta}} \in \mathbb{R}^N$ is restricted to satisfy $\overline{\theta}_i = 0$ for all $i \notin S$ and
\begin{equation*}
    \overline{\bm{\theta}}(S) = \bm{\eta}.
\end{equation*}
Last, the vector $\bm{\eta} \in \mathbb{R}^r$ is optimized by solving
\begin{equation*}
    \min_{\bm{\eta} \in \mathbb{R}^r} \tfrac{1}{N} \lVert \bm{K}(:,S) \bm{\eta} - \bm{b} \rVert^2
    + \gamma\, \bm{\eta}^T \bm{K}(S, S) \bm{\eta},
\end{equation*}
which is equivalent to the linear system
\begin{equation*}
    [\Kb(:,S)^T \Kb(:,S) + \gamma N \Kb(S,S)]\, \bm{\eta} = \Kb(:,S)^T \bb.
\end{equation*}
It takes just $\mathcal{O}(N r^2)$ operations to form and solve this linear system by a direct method.
Moreover, there is no need to generate the complete kernel matrix $\bm{K} \in \mathbb{R}^{N \times N}$; it suffices to generate the $r$-column submatrix $\bm{K}(:, S)$, and the storage requirements are thus $\mathcal{O}(N r)$.

\begin{algorithm}[t]
\caption{RPCholesky\cite{chen_randomly_2023}}\label{alg:RPCholesky}
\KwData{Formula for looking up entries of $\bm{K}$, approximation rank $r$ (multiple of $10$)\;}
\KwResult{Index set $S$\;} 

Initialize:
$\Fb \gets \Ob$, $S \gets \emptyset$, $T \gets r/10$, and $\db \gets \textup{diag}(\Kb)$\;
\qquad \textcolor{blue}{\tcp{Initialize parameters}}

\For{\textup{$i=0$ to $9$}}{
$s_{iT+1},\ldots, s_{iT+T} \sim \db/\sum_{j=1} d_j$\;
\qquad \textcolor{blue}{\tcp{Randomly sample indices}}

$S' \gets \texttt{Unique}(\{s_{iT+1},\ldots,s_{iT+T}\})$\;
\qquad \textcolor{blue}{\tcp{Identify unique indices}}

$S \gets S \cup S'$\;
\qquad \textcolor{blue}{\tcp{Add new indices to landmark set}}

$\Gb \gets \Kb(:,S') - \Fb\Fb(S',:)^\ast$\;
\qquad \textcolor{blue}{\tcp{Evaluate new columns}}

$\Gb(S',:) \gets \Gb(S',:) + \epsilon_{\rm mach} \textup{tr}(\Gb(S',:)) \Ib$ \;
\qquad \textcolor{blue}{\tcp{Regularize}}

$\Rb \gets \texttt{Cholesky}(\Gb(S',:))$\;
\qquad \textcolor{blue}{\tcp{Upper triangular Cholesky factor}}

$\Fb(:,iT+1:iT+|S'|) \gets \Gb \Rb^{-1}$\;
\qquad \textcolor{blue}{\tcp{Update approximation}}

$\db \gets \db - \texttt{SquaredRowNorms}(\Gb \Rb^{-1})$\;
\qquad \textcolor{blue}{\tcp{Update sampling probabilities}}

$\db \gets \max\{\db,0\}$\;
\qquad \textcolor{blue}{\tcp{Ensure nonnegative probabilities}}

$\db(S') \gets 0$\;
\qquad \textcolor{blue}{\tcp{Prevent double sampling}}
}

\KwRet{$S$}

\end{algorithm}

\subsection{Optimization of the hyperparameters}

The last parameters to optimize are the bandwidth $\varepsilon > 0$, the regularization $\gamma > 0$, and the approximation rank $r$.
To select these parameters, a simple but effective approach is a grid search.
In the grid search, $20\%$ of the data is set aside as validation data and the rest is training data.
The FCM is optimized using the training data.
Then the loss function \eqref{eq:loss} is evaluated using the validation data, and the best parameters are the ones that minimize the loss.

Based on the grid search results for the numerical experiments in Sec.~\ref{sec:experiments}, a good default bandwidth is $\varepsilon = 1$.
This bandwidth is intuitively reasonable since it corresponds to one standard deviation unit for the transformed data points $\bm{M}^{1/2} \bm{x}_1, \ldots, \bm{M}^{1/2} \bm{x}_n$.

Based on the grid search results, the experiments use a regularization parameter of $\gamma = 10^{-6}$, but this value is not intuitive and it remains unclear whether this would be a good default for other problems.
Any value of $\gamma$ smaller than $\gamma = 10^{-6}$ also leads to a similar loss (less than $1\%$ change).
To be conservative, the highest value $\gamma = 10^{-6}$ was selected.

The last parameter to optimize is the approximation rank $r$ that is used in RPCholesky.
Raising $r$ increases the accuracy but also increases the linear algebra cost since the FCM optimization requires $\mathcal{O}(N r^2)$ floating point operations.
Typical values of $r$ range from $10^2$--$10^4$, and the theoretical optimum depends on the number of large eigenvalues in the full-data kernel matrix \cite{chen_randomly_2023}.
In the Sec.~\ref{sec:experiments} experiments, the rank is set to $r = 10^3$, at which point the results are nearly converged (see Fig.~\ref{fig:fcm_v_nn_loss}).

\section{Numerical results} \label{sec:experiments}

This section describes numerical results from applying the fast committor machine (FCM) to two Markovian systems: the overdamped Langevin system with a triple-well potential (Sec.~\ref{sec:triple}) and a stochastic simulation of alanine dipeptide (Sec.~\ref{sec:adp}).
The code to run the experiments is available on Github\footnote{https://github.com/davidaristoff/Fast-Committor-Machine/}.

\subsection{Neural network comparison} \label{sec:comparison}

In each experiment, the FCM was compared against a fully connected feedforward neural network.
To make the comparisons fair, the models were designed with nearly the same number of parameters.
The FCM has $r$ linear coefficients and $d(d + 1)/2$ free parameters in the scaling matrix ($d$ is the phase space dimension).
This adds up to $1055$ parameters for the triple-well experiment ($r = 1000,\, d = 10$) and $1465$ parameters for the alanine dipeptide experiment ($r = 1000,\, d = 30$).
The neural architecture includes $\ell$ hidden layers with $p$ neurons per layer and a $\tanh$ nonlinearity, together with an outer layer using a sigmoid nonlinearity.
The total number of parameters in the neural network is thus
\begin{equation*}
    p (d+1) + (\ell-1) p (p + 1) + (p + 1)
\end{equation*}
and includes $1081$ parameters for the triple-well experiment ($\ell = 2,\, p = 27,\, d = 10$) and $1481$ parameters for the alanine dipeptide experiment ($\ell = 3,\, p = 20,\, d = 30$).
Larger neural nets could potentially improve the accuracy \cite{belkin2019reconciling}, but they would be more challenging and costly to train.

The neural nets were trained using the PyTorch package \cite{paszke2019pytorch} and the AdamW optimizer \cite{loshchilov2019decoupled}.
Before the training,
$20\%$ of the data was set aside for validation, and the remaining $80\%$ was training data.
During each epoch, the optimizer evaluated all the training data points and applied a sequence of stochastic gradient updates using mini-batches of $500$ points.
After each epoch, the validation data was used to estimate the loss function \eqref{eq:loss}.
After $20$ epochs with no improvement to the loss function, the training was halted and the parameter set leading to the lowest loss function was selected.

One crucial parameter when training neural nets is the learning rate, which was set to $10^{-4}$ for the triple-well experiment and $5 \times 10^{-4}$ for the alanine dipeptide experiment.
When the learning rate is too small, the training is excessively slow.
When the learning rate is too large, the model exhibits a significant decrease in accuracy or even fails to converge.
As a complication, changing the learning rate may require changing the patience parameter (number of epochs with no improvement before halting the training).

On the whole, training neural networks for committor approximation involves choosing an appropriate architecture and an appropriate learning rate with minimal guidance about the best choice of these parameters.
In contrast, the parameter-tuning for the FCM is more straightforward: the default bandwidth $\epsilon=1$ is sufficient for these experiments. Moreover, 
adjusting the approximation rank $r$ or regularization parameter $\gamma$ has direct and predictable impacts on training speed and performance.

\subsection{Triple-well potential energy} \label{sec:triple}

The first numerical experiment is based on the overdamped Langevin dynamics
\begin{equation}
\label{eq:overdamped}
    d{\xb}(t) = -\nabla V(\xb(t))\,dt + \sqrt{2\beta^{-1}}\,d{\boldsymbol w}(t).
\end{equation}
The potential function $V(x)$ is constructed as
\begin{equation*}
V(\xb) = V_0(\eb_1^T \xb, \eb_2^T \xb) + 2 \sum\nolimits_{i=3}^{10} (\eb_i^T \xb)^2,
\end{equation*}
where $V_0$ is the three-well potential~\cite{cerou_multiple_2011} that is
illustrated in Figure~\ref{fig:potential}, and $\eb_i$ is the $i$th basis vector.
The equilibrium density for this dynamics is the Gibbs measure
\begin{equation*}
    \mu(\bm{x}) = \frac{{\rm e}^{-\beta V(\bm{x})}}{\int {\rm e}^{-\beta V(\bm{y})} \mathop{d\bm{y}}},
\end{equation*}
where $\beta > 0$ is the inverse-temperature parameter.

\begin{figure}[t]
  \centering
   \subfloat[][]{\includegraphics[width=.23\textwidth]{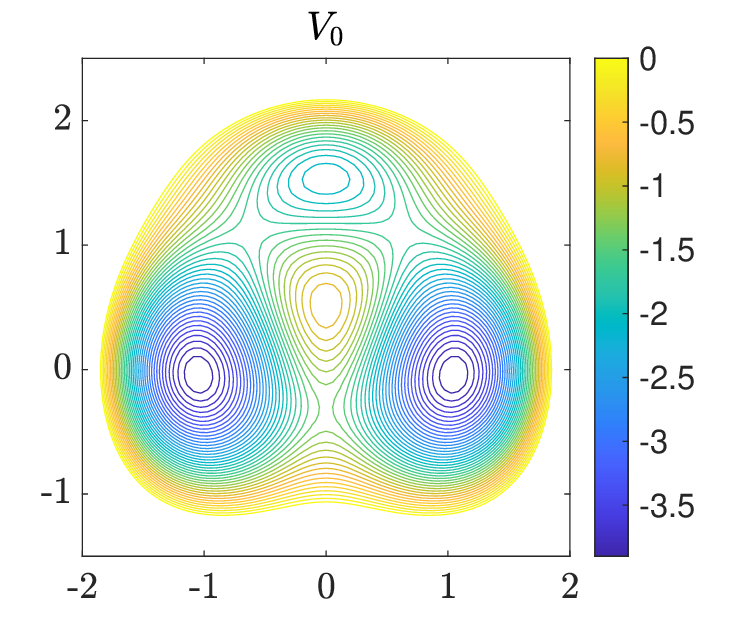}}  \quad 
   \subfloat[][]{\includegraphics[width=.23\textwidth]{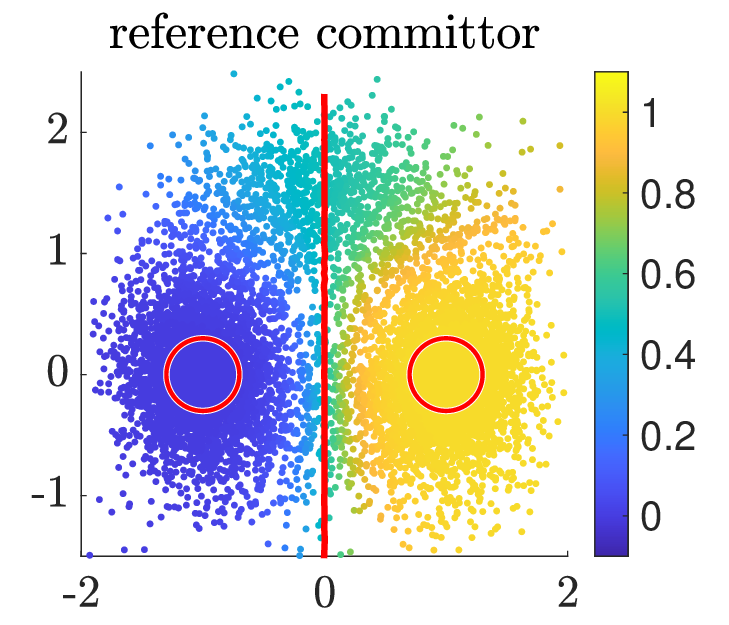}} 
      \caption{(a) The potential function $V_0$. (b) The reference committor evaluated on the validation data points with states $A$ and $B$ and the committor one-half surface indicated in red.}
  \label{fig:potential}
\end{figure}

The goal of the experiment is to calculate the committor function associated with the inverse temperature $\beta = 2$ and the states
\begin{align*}
    A &= \{\bm{x} \in \mathbb{R}^{10} \colon (\bm{e}_1^T \bm{x} + 1)^2 + (\bm{e}_2^T \bm{x})^2 \leq 0.3^2 \}, \\
    B &= \{\bm{x} \in \mathbb{R}^{10} \colon (\bm{e}_1^T \bm{x} - 1)^2 + (\bm{e}_2^T \bm{x})^2 \leq 0.3^2 \}.
\end{align*}
The definitions for $A$ and $B$ depend only on coordinates $1$ and $2$.
Coordinates $3$--$10$ are nuisance coordinates which have no effect on the committor but increase difficulty of the committor approximation.
Because the committor only depends on coordinates $1$ and $2$, the finite elements method can solve the two-dimensional PDE formulation of the committor problem to generate a highly accurate reference;
see Figure \ref{fig:potential} for an illustration.

The data set was generated by a two-stage process.
In the first stage, the initial states $(\bm{x}_n)_{1 \leq n \leq N}$ were sampled by running the Langevin dynamics \eqref{eq:overdamped} at the inverse temperature $\beta_s = 1$ and storing the positions after $N = 10^6$ uniformly spaced time intervals.
The low $\beta_s$ value was needed to ensure adequate phase space coverage.
Since the sampling distribution diverged from the target distribution, the initial states were weighted according to 
$w_n = {\rm e}^{(\beta_s - \beta) V(\xb_n)}$.
This weight definition is consistent with eq.~\eqref{eq:alternative}, which states that the weights are likelihood ratios multiplied by an arbitrary constant factor.

In the second stage, the model was run forward for an additional $\tau = 10^{-2}$  time units at the target inverse temperature $\beta = 2$, starting from $\bm{x}(0) = \bm{x}_n$ and ending at a new point $\bm{x}(\tau) = \bm{y}_n$.
This process was repeated for each data point for $n = 1, \ldots, 10^6$.

\begin{figure}[t]
  \centering
  \subfloat[][]{\includegraphics[width=.23\textwidth]{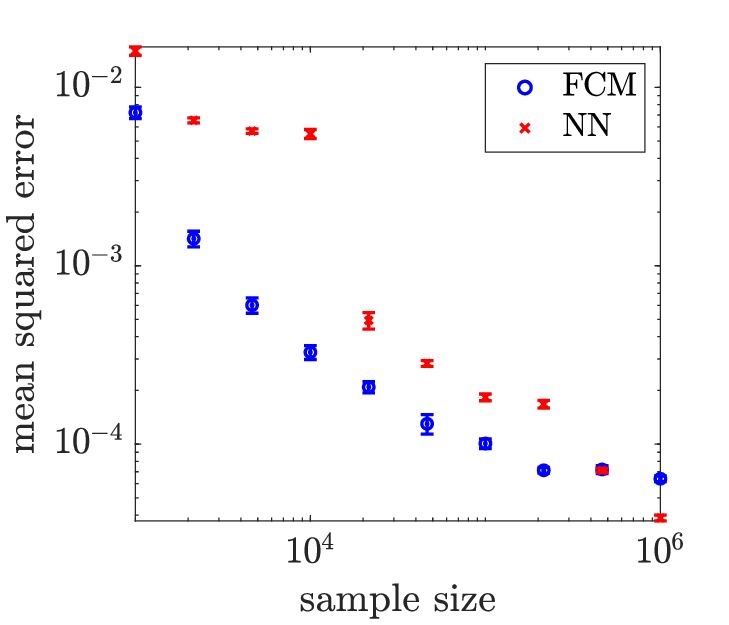}} \quad
  \subfloat[][]{\includegraphics[width=.23\textwidth]{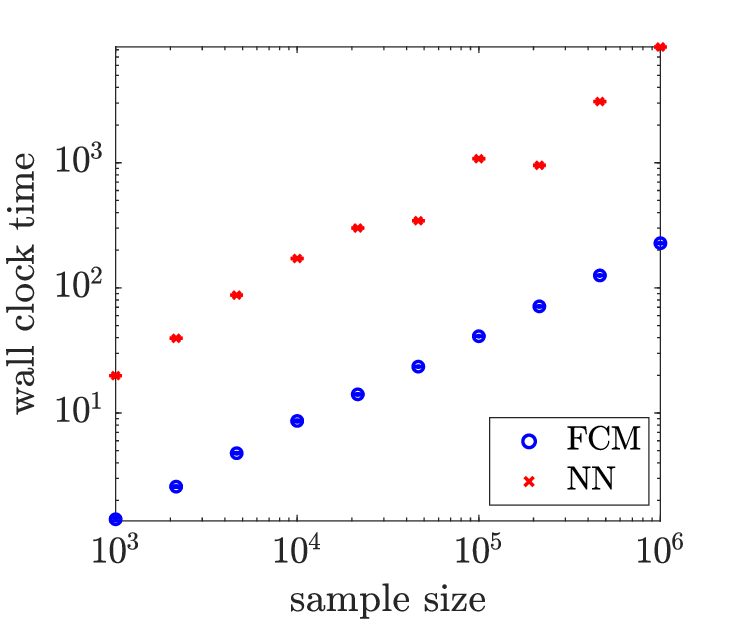}}
      \caption{Comparison of neural net (NN) and FCM performance for the triple-well system, with standard error bars computed from $10$ independent runs of the FCM. (a) Mean squared error computed using the reference committor. (b) Runtime in seconds.}
  \label{fig:toy_comparison}
\end{figure}

Figure~\ref{fig:toy_comparison} evaluate the performance of the FCM and the feedforward neural network across 10 data sizes logarithmically spaced between $N = 10^3$ and $N = 10^6$.
The largest experiments use the full data set with $N = 10^6$ data points, while the other experiments use data pairs chosen uniformly at random.
For all sample sizes $N < 10^6$, the FCM achieves higher accuracy than the neural net and also trains more quickly.

The detailed runtime and accuracy comparisons between the FCM and the neural net committor approximation may depend on implementation choices (neural net structure, stopping criteria, optimization method, etc.).
Nonetheless, these tests suggest that the FCM is a competitive method for the triple-well experiment.

As an additional advantage, the FCM exhibits robustness during training.
Figure~\ref{fig:fcm_v_nn_loss} shows that the
FCM error decreases and then stabilizes after a few iterations.
In contrast, the neural net error behaves unpredictably with epochs unless the learning rate is very small.
To interpret the plot, recall that the FCM updates the scaling matrix $\bm{M}$ once per iteration and runs for $5$ iterations, while the neural net updates the neural net parameters many times per epoch and runs for a variable number of epochs based on the stopping rule.

\begin{figure}[t]
    \centering
    \subfloat[][]{\includegraphics[width=.24\textwidth]{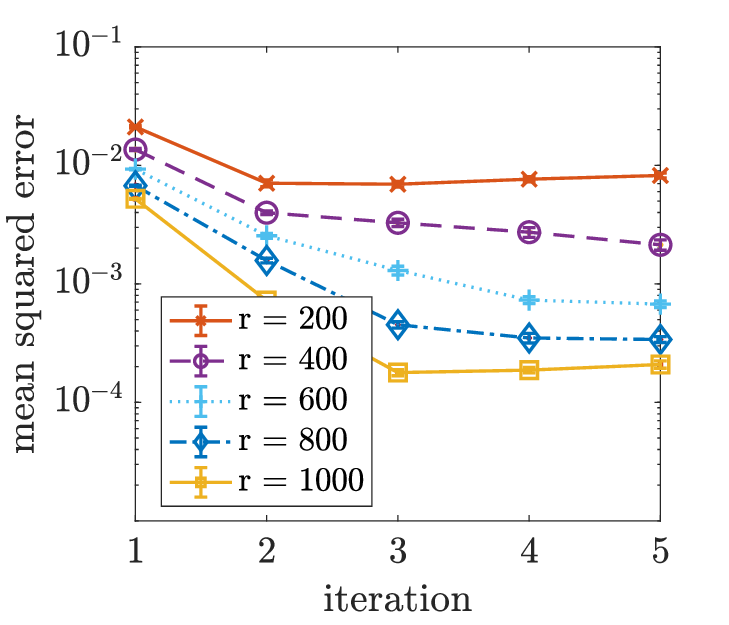}} \hfill
    \subfloat[][]{\includegraphics[width=.24\textwidth]{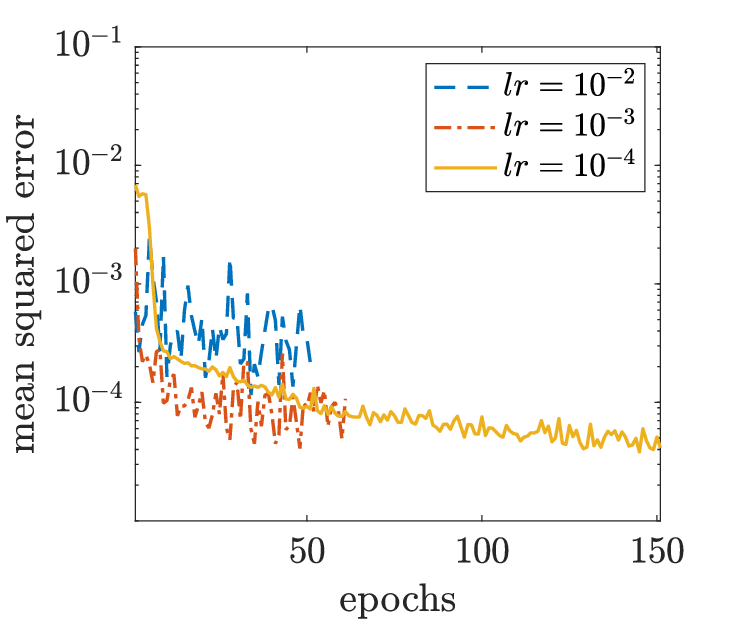}}
    \caption{Training performance for single instances of the triple-well experiment, with error computed using the reference committor from Fig.~\ref{fig:potential}.
    (a) The FCM with different approximation ranks $r$.
    (b) The neural net with different learning rates $lr$.}
    \label{fig:fcm_v_nn_loss}
\end{figure}

Last, Figure~\ref{fig:toy_M} displays
the square root of the scaling matrix, which shows the most significant subspaces for committor estimation.
After convergence, the transformation $\bm{M}^{1/2}$ maps away the nuisance coordinates $3$--$10$.
The leading eigenvector of $\bm{M}^{1/2}$ nearly aligns with coordinate $1$, signaling that coordinate $1$ is the most essential and coordinate $2$ is the second most essential.
While the two-dimensional figures make it seem that the FCM is solving an easy problem in two-dimensional space,
the FCM is actually solving a harder problem in ten-dimensional space.
Nonetheless, the FCM is reducing the problem to two dimensions through the automatic identification of the active subspaces.

\begin{figure}[t]
  \centering
   \subfloat[][]{\includegraphics[width=.23\textwidth]{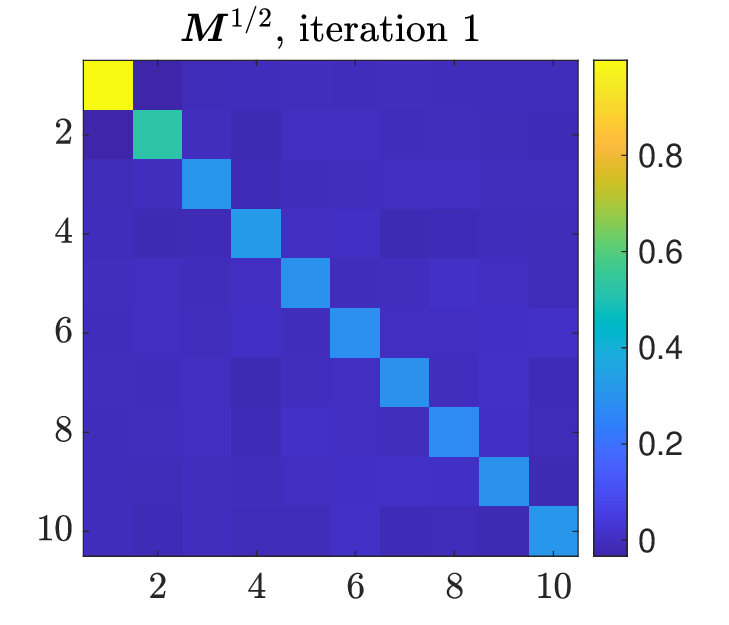}} \quad 
  \subfloat[][]{\includegraphics[width=.23\textwidth]{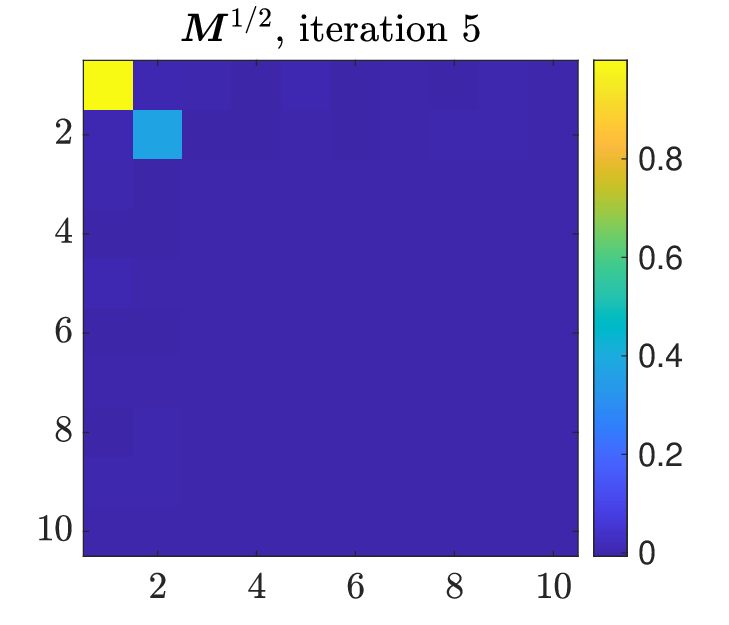}} 
      \caption{Square root of scaling matrix for the triple-well system when $N = 10^6$ and $r = 1000$. (a) After 1 iteration. (b) After 5 iterations, corresponding to convergence.}
  \label{fig:toy_M}
\end{figure}

\subsection{Alanine dipeptide} \label{sec:adp}

The previous example may seem cherry-picked for the FCM's success since there is so clearly a reduction to two dimensions.
Yet many stochastic systems can be described using low-dimensional subspaces,
even high-dimensional biomolecular systems \cite{husic2018markov}.
As a concrete example, alanine dipeptide is a small molecule whose dynamics can be reduced to a low-dimensional subspace.

The experiments in this section are drawn from a metadynamics simulation of alanine dipeptide based on the tutorial~\footnote{https://www.plumed.org/doc-v2.7/user-doc/html/masterclass-21-4.html}.
Details of the simulation can be found on Github\footnote{https://github.com/davidaristoff/Fast-Committor-Machine/}.
After excluding the hydrogen atoms, the alanine dipeptide data set contains $(x, y, z)$-coordinates for the ten backbone atoms, leading to 30-dimensional data points. 
Figure~\ref{fig:alanine_potential} shows that $\phi$ and $\psi$ dihedral angles give a simple description of the free energy surface, containing two prominent metastable states.

\begin{figure}[t]
    \centering
    \subfloat[][]
    {\includegraphics[width=.23\textwidth]{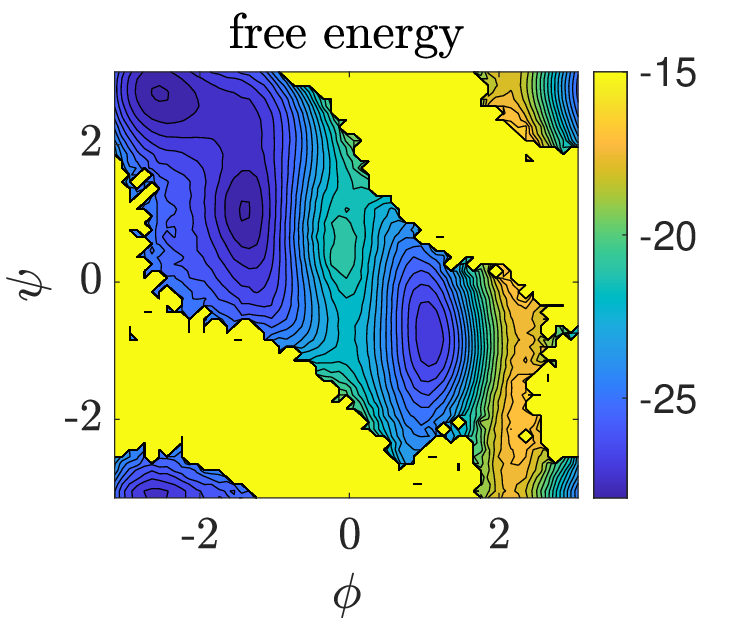}} \quad 
    \subfloat[][]
    {\includegraphics[width=.23\textwidth]{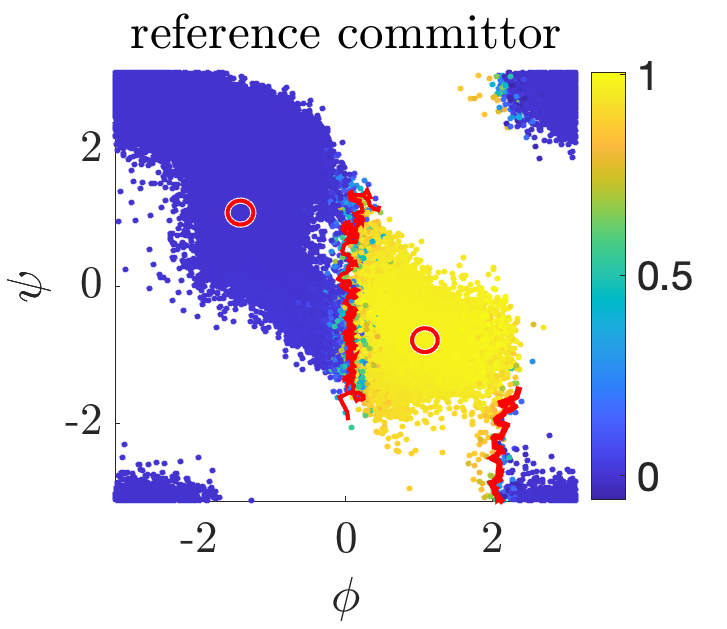}} 
    \caption{(a) Free energy surface of alanine dipeptide in $\phi$ and $\psi$ coordinates, compared to reference committor half-surface (red). (b) The reference committor in $\phi$ and $\psi$ coordinates with states $A$ and $B$ and the committor one-half surface indicated in red.}
    \label{fig:alanine_potential}
\end{figure}

The goal of the alanine dipeptide experiment is to estimate the committor function for the two metastable states, labeled as $A$ and $B$.
See Figure~\ref{fig:alanine_potential} for a picture of the precise $A$ and $B$ definitions and a reference committor which is generated by running the FCM with the largest data set ($N = 10^6$) and approximation rank $r = 2000$.

Figure~\ref{fig:alanine_MSE} shows the mean squared error and runtimes of the neural network and the FCM with $r = 1000$, as compared to the reference committor.
The results again show the FCM is more accurate than the neural network and trains more quickly. 
These results may depend on the details of the neural network optimization, but they imply the FCM is a competitive method for the alanine dipeptide experiment. 

\begin{figure}[t]
    \centering
    \subfloat[][]
    {\includegraphics[width=.23\textwidth]{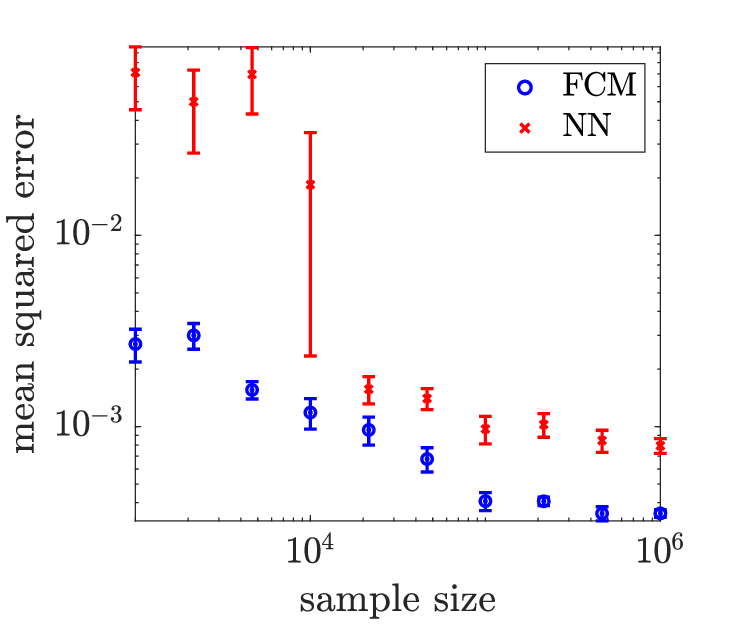}} \quad 
    \subfloat[][]
    {\includegraphics[width=.23\textwidth]{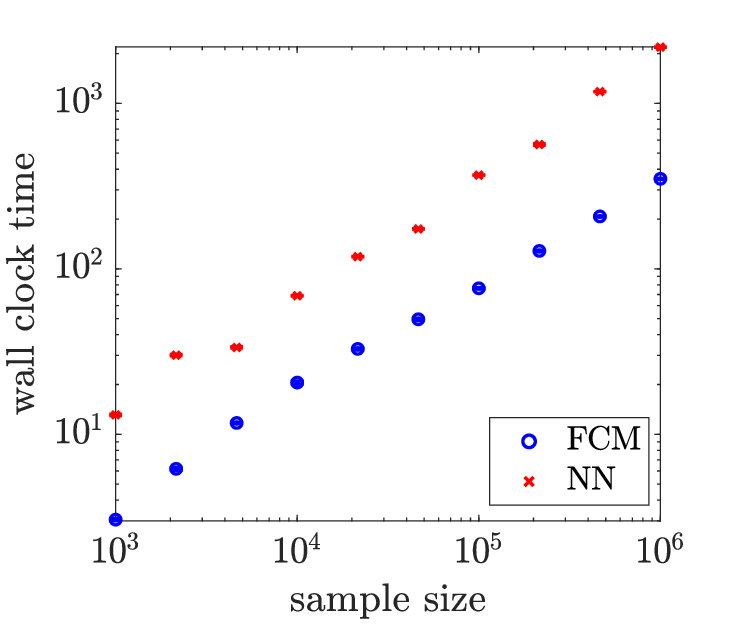}} 
    \caption{Comparison of neural net (NN) and FCM performance for alanine dipeptide, with standard error bars computed from $10$ independent simulations.
    (a) Mean squared error, computed with respect to the reference. (b) Runtime in seconds.}
    \label{fig:alanine_MSE}
\end{figure}

Last, Figure~\ref{fig:alanine_M} 
shows the square root of the FCM scaling matrix.
The matrix has two eigenvalues that are much larger than the rest, signaling that a two-dimensional linear subspace is closely aligned with the gradients of the committor.
Linear regression confirms that the top $2$ eigenvectors explain 
$95\%$ of the variance in the committor values, whereas the nonlinear $\phi$ and $\psi$ coordinates only explain $62\%$ of the variance.
See the right panel of Fig.~\ref{fig:alanine_M} for a picture of the reference committor mapped onto the top $2$ eigenvectors.

\begin{figure}[t]
    \centering
    \subfloat[][]{\includegraphics[width=.23\textwidth]{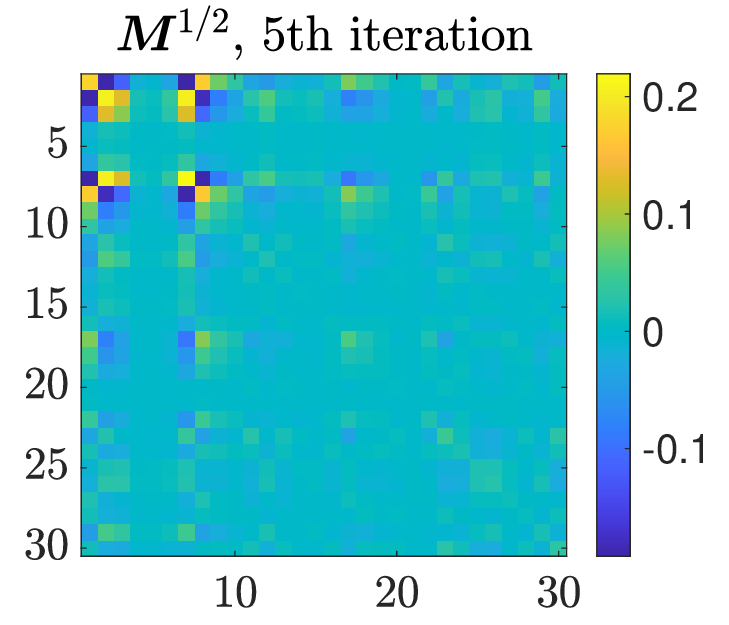}}  \quad
    \subfloat[][]{\includegraphics[width=.23\textwidth]{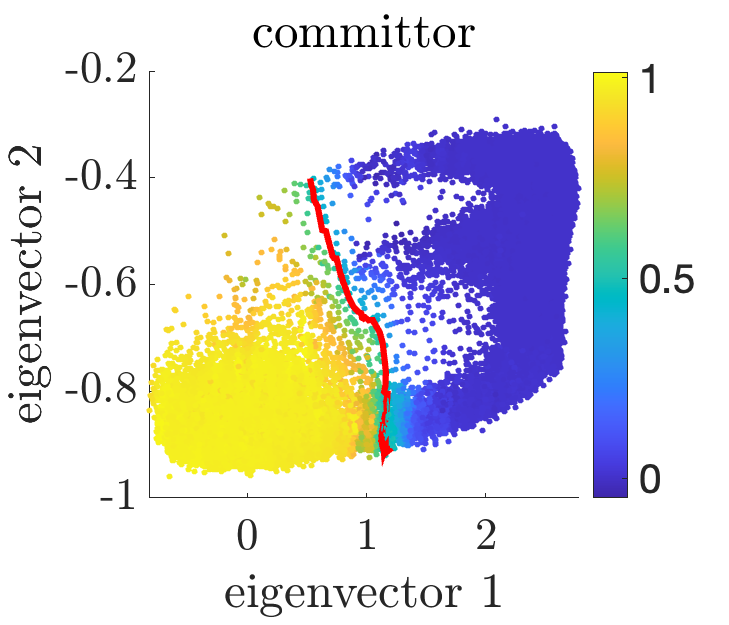}}
    \caption{Results of applying FCM to alanine dipeptide when $N = 10^6$ and $r = 1000$.
    (a) $\bm{M}^{1/2}$ after $5$ iterations of Algorithm~\ref{alg:main}, 
    corresponding to convergence.
    (b) Reference committor mapped onto top $2$ eigenvectors of $\bm{M}$ with reference committor one-half surface indicated in red.
    }
  \label{fig:alanine_M}
\end{figure}

\section{Conclusion}
\label{sec:conclusion}

The FCM is a method for efficiently solving the committor problem which shows promising results when applied to triple-well and alanine dipeptide systems.
As the main conceptual feature, the method identifies a scaling matrix that emphasizes low-dimensional subspaces with maximal variation in the committor values.
The method also uses randomized numerical linear algebra to achieve a training time faster than a neural net with the same number of parameters.

The next research goal is testing the FCM on high-dimensional stochastic systems arising in molecular dynamics and other areas of science.
In these systems, it may be challenging to generate an appropriate data set, since the states $A$ and $B$ might be separated by large free energy barriers, which make transition events rare and justify the use of enhanced sampling.
Even for the 10- and 30-dimensional experiments in this paper,
importance sampling and
metadynamics were needed to ensure a large density of sample points in the transition region where the committor is $0.1$--$0.9$.
These sampling approaches were applied heuristically, and future progress will require a more careful mathematical and empirical study.
To that end, a promising approach that was recently developed for committor approximation \cite{rotskoff2022active} involves umbrella sampling with a sequence of bins emphasizing different ranges of committor values, including many bins near the committor one-half surface.
Building on this strategy, there is hope to develop a fully adaptive FCM method that generates data and a corresponding committor estimate using only the $A$ and $B$ definitions.

Last, there remain challenging mathematical questions about the FCM's performance.
For example, why does the method take so few iterations to converge and what makes the exponential kernel \eqref{e:kMkernel} preferable over other choices?
Addditionally, there is a deep question at the core of the FCM about why the linear rescaling is so successful, how to describe it mathematically, and what nonlinear extensions would be possible.

\section*{Acknowledgements}

D. Aristoff and G. Simpson gratefully acknowledge support from the National Science Foundation via Award No. DMS 2111277.
R.J.W. was supported by the Office of Naval Research through BRC Award N00014-18-1-2363, the National Science Foundation through FRG
Award 1952777, and Caltech through the Carver Mead New Adventures Fund, under the aegis of Joel A. Tropp.

\appendix \label{sec:appendix}

\section{Justification for the kernel approximation} \label{sec:alternative}

The FCM is based on two choices regarding the shape of the kernel function and the coefficients used in the optimization.
This appendix justifies the choices that were made.

The exponential kernel \eqref{e:kMkernel} is used throughout the paper.
In the Sec.~\ref{sec:experiments} experiments, the exponential kernel outperforms the popular square exponential kernel
\begin{equation*}
    k_{\bm{M}}(\bm{x}, \bm{x}') 
    = \exp\bigl(-\varepsilon^{-2} \lVert \bm{M}^{1/2} (\bm{x} - \bm{x}') \rVert^2 \bigr).
\end{equation*}
Nonetheless, there may be an opportunity to further improve the FCM by using an alternative kernel of the form
\begin{equation*}
    k_{\bm{M}}(\bm{x}, \bm{x}') = \phi(\varepsilon^{-1}\|\bm{M}^{1/2}(\xb-\xb')\|)
\end{equation*}
for an optimized univariate function $\phi: \mathbb{R} \rightarrow \mathbb{R}$.

The most general data-driven approximation using the exponential kernel is a linear combination of kernel functions centered on the input data points $\bm{x}_n$ or the output data points $\bm{y}_n$.
The functional form can be written as
\begin{equation*}
    q_{\bm{c}, \bm{d}}(\bm{x}) 
    = \sum_{n=1}^N 
    c_n k_{\bm{M}}(\bm{x}_n, \bm{x}) 
    + \sum_{n=1}^N d_n k_{\bm{M}}(\bm{y}_n, \bm{x}),
\end{equation*}
where $\bm{c} \in \mathbb{R}^N$ and $\bm{d} \in \mathbb{R}^N$ are variational parameters to be optimized.
Yet, this form can be simplified.
Theorem \ref{thm:optimal} shows the optimal coefficients must satisfy $\bm{c} = -\bm{d}$, which leads to a more specific and concise committor approximation
\begin{equation*}
    q_{\bm{\theta}}(\bm{x}) 
    = 
    \sum_{n=1}^N 
    \theta_n \bigl[ k_{\bm{M}}(\bm{x}_n, \bm{x}) - k_{\bm{M}}(\bm{y}_n, \bm{x})].
\end{equation*}
The proof relies on direct linear algebra calculations.

\begin{theorem} \label{thm:optimal}
Define the least-squares loss function
\begin{equation*}
    L_{\gamma}(\bm{c}, \bm{d}) = \frac{1}{N} \biggl\lVert \begin{bmatrix} \mathbf{I} \\ -\mathbf{I} \end{bmatrix}^T\!
    \bm{K}\!
    \begin{bmatrix} \bm{c} \\ \bm{d} \end{bmatrix} - \bm{b} \biggr\rVert^2
    + \gamma\, \begin{bmatrix} \bm{c} \\ \bm{d} \end{bmatrix}^T\! \bm{K}\! \begin{bmatrix} \bm{c} \\ \bm{d} \end{bmatrix},
\end{equation*}
where the positive semidefinite kernel matrix
\begin{equation*}
    \bm{K} = \begin{bmatrix} \bm{K}^{11} & \bm{K}^{12} \\
    \bm{K}^{21} & \bm{K}^{22}
    \end{bmatrix},
\end{equation*}
consists of four blocks with entries
\begin{align*}
    k^{11}_{mn} &= \sqrt{w_m} \sqrt{w_n}\, k_{\bm{M}}(\bm{x}_m, \bm{x}_n), \\
    k^{12}_{mn} &= \sqrt{w_m} \sqrt{w_n}\, k_{\bm{M}}(\bm{x}_m, \bm{y}_n), \\
    k^{21}_{mn} &= \sqrt{w_m} \sqrt{w_n}\, k_{\bm{M}}(\bm{y}_m, \bm{x}_n), \\
    k^{22}_{mn} &= \sqrt{w_m} \sqrt{w_n}\, k_{\bm{M}}(\bm{y}_m, \bm{y}_n).
\end{align*}
Then, the loss function $\bm{L}_{\gamma}(\bm{c}, \bm{d})$ has a minimizer that satisfies $\bm{c} + \bm{d} = \bm{0}$.
\end{theorem}
\begin{proof}
    The loss function is convex, so any minimizers can be identified by setting the gradient equal to zero:
    \begin{equation*}
        \frac{2}{N}\, \bm{K} \begin{bmatrix} \mathbf{I} \\ -\mathbf{I} \end{bmatrix}
        \biggl[\begin{bmatrix} \mathbf{I} \\ -\mathbf{I} \end{bmatrix}^T
        \bm{K}
        \begin{bmatrix} \bm{c} \\ \bm{d} \end{bmatrix} - \bm{b}\biggr]
        + 2 \gamma\, \bm{K}
        \begin{bmatrix} \bm{c} \\ \bm{d} \end{bmatrix}
        = \bm{0}.
    \end{equation*}
    The equation can be rearranged to yield
    \begin{equation*}
        \bm{K} \biggl[ \begin{bmatrix} \mathbf{I} \\ -\mathbf{I} \end{bmatrix} \begin{bmatrix} \mathbf{I} \\ -\mathbf{I} \end{bmatrix}^T \bm{K} + \gamma N \begin{bmatrix} \bm{I} & \\
        & \bm{I} \end{bmatrix} \biggr] \begin{bmatrix} \bm{c} \\ \bm{d} \end{bmatrix} = \bm{K} \begin{bmatrix} \mathbf{I} \\ -\mathbf{I} \end{bmatrix} \bm{b}
    \end{equation*}
    Therefore, there is a minimizer which comes from choosing $\bm{c}$ and $\bm{d}$ to satisfy the positive definite linear system
    \begin{equation}
    \label{eq:the_system}
        \biggl[ \begin{bmatrix} \mathbf{I} \\ -\mathbf{I} \end{bmatrix} \begin{bmatrix} \mathbf{I} \\ -\mathbf{I} \end{bmatrix}^T \bm{K} + \gamma N \begin{bmatrix} \bm{I} & \\
        & \bm{I} \end{bmatrix} \biggr] \begin{bmatrix} \bm{c} \\ \bm{d} \end{bmatrix} = \begin{bmatrix} \mathbf{I} \\ -\mathbf{I} \end{bmatrix} \bm{b}.
    \end{equation}
    (All other minimizers come from adding a vector in the nullspace of $\bm{K}$).
    Last, multiply the system \eqref{eq:the_system} on the left by $\begin{bmatrix} \bm{I} & \bm{I} \end{bmatrix}$ to reveal $\bm{c} + \bm{d} = \bm{0}$.
\end{proof}

\section{The scaling matrix} \label{sec:scaling}

In any learning task, it can be helpful to model a function $f: \mathbb{R}^d \rightarrow \mathbb{R}$ as the composition of an invertible linear map $\bm{A}: \mathbb{R}^d \rightarrow \mathbb{R}^d$ and a function $g_{\bm{A}}$ that is more amenable to learning: 
\begin{equation*}
    f(\bm{x}) = g_{\bm{A}}(\bm{A} \bm{x}), \quad \bm{x} \in \mathbb{R}^d.
\end{equation*}
One way to ensure the learnability of the function $g_{\bm{A}}$ is by selecting the matrix $\bm{A}$ so that the gradients
\begin{equation*}
    \nabla g_{\bm{A}}(\bm{z}_1), \ldots, \nabla g_{\bm{A}}(\bm{z}_N)
\end{equation*}
are isotropic, where 
\begin{equation*}
    \bm{z}_i = \bm{A} \bm{x}_i, \quad i = 1, 2, \ldots, N
\end{equation*}
are the linearly transformed data points.
The optimal transformation can be characterized as follows:

\begin{proposition}\label{thm:anisotropy}
The following are equivalent:
\begin{itemize}
      \item[(i)] $\bm{A}$ is invertible, and the sample gradients $\nabla g_{\bm{A}}(\bm{z}_1), \ldots, \nabla g_{\bm{A}}(\bm{z}_N)$ are isotropic, that is, 
      \begin{equation*}
      \frac{1}{N} \sum_{i=1}^N 
      |\bm{u}^T \nabla g_{\bm{A}}(\bm{z}_i)|^2 = 1
      \end{equation*}
      for any unit vector $\ub \in \mathbb{R}^d$.
      \item[(ii)] The average gradient product
      \begin{equation*}
          \bm{M} = \frac{1}{N}\sum_{i=1}^N \nabla f(\xb_i) \nabla f(\xb_i)^T.
      \end{equation*}
      is invertible, and $\Ab = \bm{Q} \bm{M}^{1/2}$ for an orthogonal matrix $\bm{Q} \in \mathbb{R}^{d \times d}$.
      \item[(iii)] $\bm{M}$ is invertible, and $\bm{A}$ transforms distances according to
      \begin{equation*}
      \lVert \bm{A}(\bm{x} - \bm{x}') \rVert
      = \lVert \bm{M}^{1/2} (\bm{x} - \bm{x}') \rVert,
      \end{equation*}
      for each $\bm{x}, \bm{x}' \in \mathbb{R}^d$.
\end{itemize}
\end{proposition}

\begin{proof}
Either of the conditions (i)-(ii) implies the linear map $\bm{A}$ is invertible.
Therefore calculate
\begin{align*}
    \frac{1}{N}\sum_{i=1}^N |\bm{u}^T \nabla g_{\Ab}(\zb_i) |^2 
    &=  \frac{1}{N} \sum_{i=1}^N |\ub^T \Ab^{-T} \nabla f(\xb_i)|^2 \\
    &=  \ub^T \Ab^{-T} \bm{M} \Ab^{-1} \ub.
\end{align*}
The above display is $1$ for each unit vector $\ub \in \mathbb{R}^d$ if and only if $\bm{A}^{-T} \bm{M} \bm{A}^{-1} = \mathbf{I}$ and $\bm{M}^{1/2}\Ab^{-1}$ is an orthogonal matrix. 
Thus, (i) and (ii) are equivalent. 

Clearly, (ii) implies (iii).
Conversely, if $\lVert \bm{A} \bm{x} \rVert = \lVert \bm{M}^{1/2} \bm{x} \rVert$ for each $\xb \in \mathbb{R}^d$, it follows that 
\begin{equation*}
    \bm{M} = \Ab^T \Ab.
\end{equation*}
Consequently, if $\bm{A} = \bm{U} \bm{\Sigma} \bm{V}^T$
is a singular value decomposition, $\bm{M} = \bm{V} \bm{\Sigma}^2 \bm{V}^T$ is an eigenvalue decomposition and
\begin{equation}
    \bm{A} = \bm{Q} \bm{M}^{1/2}, \quad \text{where } \bm{Q} = \bm{U} \bm{V}^T \text{ is orthogonal.}
\end{equation}
This shows that (ii) and (iii) are equivalent.
\end{proof}

\bibliography{aipsamp}

\begin{thebibliography}{58}%
\makeatletter
\providecommand \@ifxundefined [1]{%
 \@ifx{#1\undefined}
}%
\providecommand \@ifnum [1]{%
 \ifnum #1\expandafter \@firstoftwo
 \else \expandafter \@secondoftwo
 \fi
}%
\providecommand \@ifx [1]{%
 \ifx #1\expandafter \@firstoftwo
 \else \expandafter \@secondoftwo
 \fi
}%
\providecommand \natexlab [1]{#1}%
\providecommand \enquote  [1]{``#1''}%
\providecommand \bibnamefont  [1]{#1}%
\providecommand \bibfnamefont [1]{#1}%
\providecommand \citenamefont [1]{#1}%
\providecommand \href@noop [0]{\@secondoftwo}%
\providecommand \href [0]{\begingroup \@sanitize@url \@href}%
\providecommand \@href[1]{\@@startlink{#1}\@@href}%
\providecommand \@@href[1]{\endgroup#1\@@endlink}%
\providecommand \@sanitize@url [0]{\catcode `\\12\catcode `\$12\catcode
  `\&12\catcode `\#12\catcode `\^12\catcode `\_12\catcode `\%12\relax}%
\providecommand \@@startlink[1]{}%
\providecommand \@@endlink[0]{}%
\providecommand \url  [0]{\begingroup\@sanitize@url \@url }%
\providecommand \@url [1]{\endgroup\@href {#1}{\urlprefix }}%
\providecommand \urlprefix  [0]{URL }%
\providecommand \Eprint [0]{\href }%
\providecommand \doibase [0]{http://dx.doi.org/}%
\providecommand \selectlanguage [0]{\@gobble}%
\providecommand \bibinfo  [0]{\@secondoftwo}%
\providecommand \bibfield  [0]{\@secondoftwo}%
\providecommand \translation [1]{[#1]}%
\providecommand \BibitemOpen [0]{}%
\providecommand \bibitemStop [0]{}%
\providecommand \bibitemNoStop [0]{.\EOS\space}%
\providecommand \EOS [0]{\spacefactor3000\relax}%
\providecommand \BibitemShut  [1]{\csname bibitem#1\endcsname}%
\let\auto@bib@innerbib\@empty
\bibitem [{\citenamefont {Newhall}\ and\ \citenamefont
  {Vanden-Eijnden}(2017)}]{newhall_metastability_2017}%
  \BibitemOpen
  \bibfield  {author} {\bibinfo {author} {\bibfnamefont {K.~A.}\ \bibnamefont
  {Newhall}}\ and\ \bibinfo {author} {\bibfnamefont {E.}~\bibnamefont
  {Vanden-Eijnden}},\ }\href {\doibase 10.1007/s00332-016-9358-x} {\bibfield
  {journal} {\bibinfo  {journal} {Journal of Nonlinear Science}\ }\textbf
  {\bibinfo {volume} {27}},\ \bibinfo {pages} {1007} (\bibinfo {year}
  {2017})}\BibitemShut {NoStop}%
\bibitem [{\citenamefont {Grafke}\ and\ \citenamefont
  {Vanden-Eijnden}(2019)}]{grafke_numerical_2019}%
  \BibitemOpen
  \bibfield  {author} {\bibinfo {author} {\bibfnamefont {T.}~\bibnamefont
  {Grafke}}\ and\ \bibinfo {author} {\bibfnamefont {E.}~\bibnamefont
  {Vanden-Eijnden}},\ }\href {\doibase 10.1063/1.5084025} {\bibfield  {journal}
  {\bibinfo  {journal} {Chaos: An Interdisciplinary Journal of Nonlinear
  Science}\ }\textbf {\bibinfo {volume} {29}},\ \bibinfo {pages} {063118}
  (\bibinfo {year} {2019})}\BibitemShut {NoStop}%
\bibitem [{\citenamefont {Jacques-Dumas}\ \emph {et~al.}(2023)\citenamefont
  {Jacques-Dumas}, \citenamefont {van Westen}, \citenamefont {Bouchet},\ and\
  \citenamefont {Dijkstra}}]{jacques-dumas_data-driven_2023}%
  \BibitemOpen
  \bibfield  {author} {\bibinfo {author} {\bibfnamefont {V.}~\bibnamefont
  {Jacques-Dumas}}, \bibinfo {author} {\bibfnamefont {R.~M.}\ \bibnamefont {van
  Westen}}, \bibinfo {author} {\bibfnamefont {F.}~\bibnamefont {Bouchet}}, \
  and\ \bibinfo {author} {\bibfnamefont {H.~A.}\ \bibnamefont {Dijkstra}},\
  }\href {\doibase 10.5194/npg-30-195-2023} {\bibfield  {journal} {\bibinfo
  {journal} {Nonlinear Processes in Geophysics}\ }\textbf {\bibinfo {volume}
  {30}},\ \bibinfo {pages} {195} (\bibinfo {year} {2023})}\BibitemShut
  {NoStop}%
\bibitem [{\citenamefont {Lucente}\ \emph
  {et~al.}(2022{\natexlab{a}})\citenamefont {Lucente}, \citenamefont
  {Herbert},\ and\ \citenamefont {Bouchet}}]{Lucente_Herbert_Bouchet_2022}%
  \BibitemOpen
  \bibfield  {author} {\bibinfo {author} {\bibfnamefont {D.}~\bibnamefont
  {Lucente}}, \bibinfo {author} {\bibfnamefont {C.}~\bibnamefont {Herbert}}, \
  and\ \bibinfo {author} {\bibfnamefont {F.}~\bibnamefont {Bouchet}},\ }\href
  {\doibase 10.1175/JAS-D-22-0038.1} {\bibfield  {journal} {\bibinfo  {journal}
  {Journal of the Atmospheric Sciences}\ }\textbf {\bibinfo {volume} {79}},\
  \bibinfo {pages} {2387} (\bibinfo {year} {2022}{\natexlab{a}})}\BibitemShut
  {NoStop}%
\bibitem [{\citenamefont {Bolhuis}\ \emph {et~al.}(2002)\citenamefont
  {Bolhuis}, \citenamefont {Chandler}, \citenamefont {Dellago},\ and\
  \citenamefont {Geissler}}]{bolhuis_transition_2002}%
  \BibitemOpen
  \bibfield  {author} {\bibinfo {author} {\bibfnamefont {P.~G.}\ \bibnamefont
  {Bolhuis}}, \bibinfo {author} {\bibfnamefont {D.}~\bibnamefont {Chandler}},
  \bibinfo {author} {\bibfnamefont {C.}~\bibnamefont {Dellago}}, \ and\
  \bibinfo {author} {\bibfnamefont {P.~L.}\ \bibnamefont {Geissler}},\ }\href
  {\doibase 10.1146/annurev.physchem.53.082301.113146} {\bibfield  {journal}
  {\bibinfo  {journal} {Annual Review of Physical Chemistry}\ }\textbf
  {\bibinfo {volume} {53}},\ \bibinfo {pages} {291} (\bibinfo {year}
  {2002})}\BibitemShut {NoStop}%
\bibitem [{\citenamefont {Rogal}\ and\ \citenamefont
  {Bolhuis}(2008)}]{rogal_multiple_2008}%
  \BibitemOpen
  \bibfield  {author} {\bibinfo {author} {\bibfnamefont {J.}~\bibnamefont
  {Rogal}}\ and\ \bibinfo {author} {\bibfnamefont {P.~G.}\ \bibnamefont
  {Bolhuis}},\ }\href {\doibase 10.1063/1.3029696} {\bibfield  {journal}
  {\bibinfo  {journal} {The Journal of Chemical Physics}\ }\textbf {\bibinfo
  {volume} {129}},\ \bibinfo {pages} {224107} (\bibinfo {year}
  {2008})}\BibitemShut {NoStop}%
\bibitem [{\citenamefont {Allen}\ \emph {et~al.}(2006)\citenamefont {Allen},
  \citenamefont {Frenkel},\ and\ \citenamefont {ten
  Wolde}}]{allen_forward_2006}%
  \BibitemOpen
  \bibfield  {author} {\bibinfo {author} {\bibfnamefont {R.~J.}\ \bibnamefont
  {Allen}}, \bibinfo {author} {\bibfnamefont {D.}~\bibnamefont {Frenkel}}, \
  and\ \bibinfo {author} {\bibfnamefont {P.~R.}\ \bibnamefont {ten Wolde}},\
  }\href {\doibase 10.1063/1.2198827} {\bibfield  {journal} {\bibinfo
  {journal} {The Journal of Chemical Physics}\ }\textbf {\bibinfo {volume}
  {124}},\ \bibinfo {pages} {194111} (\bibinfo {year} {2006})}\BibitemShut
  {NoStop}%
\bibitem [{\citenamefont {Escobedo}\ \emph {et~al.}(2009)\citenamefont
  {Escobedo}, \citenamefont {Borrero},\ and\ \citenamefont
  {Araque}}]{escobedo_transition_2009}%
  \BibitemOpen
  \bibfield  {author} {\bibinfo {author} {\bibfnamefont {F.~A.}\ \bibnamefont
  {Escobedo}}, \bibinfo {author} {\bibfnamefont {E.~E.}\ \bibnamefont
  {Borrero}}, \ and\ \bibinfo {author} {\bibfnamefont {J.~C.}\ \bibnamefont
  {Araque}},\ }\href {\doibase 10.1088/0953-8984/21/33/333101} {\bibfield
  {journal} {\bibinfo  {journal} {Journal of Physics: Condensed Matter}\
  }\textbf {\bibinfo {volume} {21}},\ \bibinfo {pages} {333101} (\bibinfo
  {year} {2009})}\BibitemShut {NoStop}%
\bibitem [{\citenamefont {E}\ \emph {et~al.}(2005)\citenamefont {E},
  \citenamefont {Ren},\ and\ \citenamefont {Vanden-Eijnden}}]{e_finite_2005}%
  \BibitemOpen
  \bibfield  {author} {\bibinfo {author} {\bibfnamefont {W.}~\bibnamefont {E}},
  \bibinfo {author} {\bibfnamefont {W.}~\bibnamefont {Ren}}, \ and\ \bibinfo
  {author} {\bibfnamefont {E.}~\bibnamefont {Vanden-Eijnden}},\ }\href
  {\doibase 10.1021/jp0455430} {\bibfield  {journal} {\bibinfo  {journal} {The
  Journal of Physical Chemistry B}\ }\textbf {\bibinfo {volume} {109}},\
  \bibinfo {pages} {6688} (\bibinfo {year} {2005})}\BibitemShut {NoStop}%
\bibitem [{\citenamefont {Vanden-Eijnden}\ and\ \citenamefont
  {Venturoli}(2009)}]{vanden-eijnden_revisiting_2009}%
  \BibitemOpen
  \bibfield  {author} {\bibinfo {author} {\bibfnamefont {E.}~\bibnamefont
  {Vanden-Eijnden}}\ and\ \bibinfo {author} {\bibfnamefont {M.}~\bibnamefont
  {Venturoli}},\ }\href {\doibase 10.1063/1.3130083} {\bibfield  {journal}
  {\bibinfo  {journal} {The Journal of Chemical Physics}\ }\textbf {\bibinfo
  {volume} {130}},\ \bibinfo {pages} {194103} (\bibinfo {year}
  {2009})}\BibitemShut {NoStop}%
\bibitem [{\citenamefont {Cérou}\ \emph {et~al.}(2011)\citenamefont {Cérou},
  \citenamefont {Guyader}, \citenamefont {Lelièvre},\ and\ \citenamefont
  {Pommier}}]{cerou_multiple_2011}%
  \BibitemOpen
  \bibfield  {author} {\bibinfo {author} {\bibfnamefont {F.}~\bibnamefont
  {Cérou}}, \bibinfo {author} {\bibfnamefont {A.}~\bibnamefont {Guyader}},
  \bibinfo {author} {\bibfnamefont {T.}~\bibnamefont {Lelièvre}}, \ and\
  \bibinfo {author} {\bibfnamefont {D.}~\bibnamefont {Pommier}},\ }\href
  {\doibase 10.1063/1.3518708} {\bibfield  {journal} {\bibinfo  {journal} {The
  Journal of Chemical Physics}\ }\textbf {\bibinfo {volume} {134}},\ \bibinfo
  {pages} {054108} (\bibinfo {year} {2011})}\BibitemShut {NoStop}%
\bibitem [{\citenamefont {Khoo}\ \emph {et~al.}(2019)\citenamefont {Khoo},
  \citenamefont {Lu},\ and\ \citenamefont {Ying}}]{khoo_solving_2019}%
  \BibitemOpen
  \bibfield  {author} {\bibinfo {author} {\bibfnamefont {Y.}~\bibnamefont
  {Khoo}}, \bibinfo {author} {\bibfnamefont {J.}~\bibnamefont {Lu}}, \ and\
  \bibinfo {author} {\bibfnamefont {L.}~\bibnamefont {Ying}},\ }\href {\doibase
  10.1007/s40687-018-0160-2} {\bibfield  {journal} {\bibinfo  {journal}
  {Research in the Mathematical Sciences}\ }\textbf {\bibinfo {volume} {6}},\
  \bibinfo {pages} {1} (\bibinfo {year} {2019})}\BibitemShut {NoStop}%
\bibitem [{\citenamefont {Li}\ \emph {et~al.}(2019)\citenamefont {Li},
  \citenamefont {Lin},\ and\ \citenamefont {Ren}}]{Li2019}%
  \BibitemOpen
  \bibfield  {author} {\bibinfo {author} {\bibfnamefont {Q.}~\bibnamefont
  {Li}}, \bibinfo {author} {\bibfnamefont {B.}~\bibnamefont {Lin}}, \ and\
  \bibinfo {author} {\bibfnamefont {W.}~\bibnamefont {Ren}},\ }\href {\doibase
  10.1063/1.5110439} {\bibfield  {journal} {\bibinfo  {journal} {The Journal of
  Chemical Physics}\ }\textbf {\bibinfo {volume} {151}},\ \bibinfo {pages}
  {054112} (\bibinfo {year} {2019})}\BibitemShut {NoStop}%
\bibitem [{\citenamefont {Li}\ \emph {et~al.}(2022)\citenamefont {Li},
  \citenamefont {Khoo}, \citenamefont {Ren},\ and\ \citenamefont
  {Ying}}]{li_semigroup_2022}%
  \BibitemOpen
  \bibfield  {author} {\bibinfo {author} {\bibfnamefont {H.}~\bibnamefont
  {Li}}, \bibinfo {author} {\bibfnamefont {Y.}~\bibnamefont {Khoo}}, \bibinfo
  {author} {\bibfnamefont {Y.}~\bibnamefont {Ren}}, \ and\ \bibinfo {author}
  {\bibfnamefont {L.}~\bibnamefont {Ying}},\ }in\ \href
  {https://proceedings.mlr.press/v145/li22a.html} {\emph {\bibinfo {booktitle}
  {Proceedings of the 2nd Mathematical and Scientific Machine Learning
  Conference}}}\ (\bibinfo {year} {2022})\BibitemShut {NoStop}%
\bibitem [{\citenamefont {Evans}\ \emph {et~al.}(2022)\citenamefont {Evans},
  \citenamefont {Cameron},\ and\ \citenamefont
  {Tiwary}}]{evans_computing_2022}%
  \BibitemOpen
  \bibfield  {author} {\bibinfo {author} {\bibfnamefont {L.}~\bibnamefont
  {Evans}}, \bibinfo {author} {\bibfnamefont {M.~K.}\ \bibnamefont {Cameron}},
  \ and\ \bibinfo {author} {\bibfnamefont {P.}~\bibnamefont {Tiwary}},\ }\href
  {\doibase 10.1063/5.0122990} {\bibfield  {journal} {\bibinfo  {journal} {The
  Journal of Chemical Physics}\ }\textbf {\bibinfo {volume} {157}},\ \bibinfo
  {pages} {214107} (\bibinfo {year} {2022})}\BibitemShut {NoStop}%
\bibitem [{\citenamefont {Lai}\ and\ \citenamefont
  {Lu}(2018)}]{lai_point_2018}%
  \BibitemOpen
  \bibfield  {author} {\bibinfo {author} {\bibfnamefont {R.}~\bibnamefont
  {Lai}}\ and\ \bibinfo {author} {\bibfnamefont {J.}~\bibnamefont {Lu}},\
  }\href {\doibase 10.1137/17M1123018} {\bibfield  {journal} {\bibinfo
  {journal} {Multiscale Modeling \& Simulation}\ }\textbf {\bibinfo {volume}
  {16}},\ \bibinfo {pages} {710} (\bibinfo {year} {2018})}\BibitemShut
  {NoStop}%
\bibitem [{\citenamefont {Lucente}\ \emph
  {et~al.}(2022{\natexlab{b}})\citenamefont {Lucente}, \citenamefont {Rolland},
  \citenamefont {Herbert},\ and\ \citenamefont
  {Bouchet}}]{lucente_coupling_2021}%
  \BibitemOpen
  \bibfield  {author} {\bibinfo {author} {\bibfnamefont {D.}~\bibnamefont
  {Lucente}}, \bibinfo {author} {\bibfnamefont {J.}~\bibnamefont {Rolland}},
  \bibinfo {author} {\bibfnamefont {C.}~\bibnamefont {Herbert}}, \ and\
  \bibinfo {author} {\bibfnamefont {F.}~\bibnamefont {Bouchet}},\ }\href
  {\doibase 10.1088/1742-5468/ac7aa7} {\bibfield  {journal} {\bibinfo
  {journal} {Journal of Statistical Mechanics: Theory and Experiment}\ }\textbf
  {\bibinfo {volume} {2022}},\ \bibinfo {pages} {083201} (\bibinfo {year}
  {2022}{\natexlab{b}})}\BibitemShut {NoStop}%
\bibitem [{\citenamefont {Berezhkovskii}\ and\ \citenamefont
  {Szabo}(2019)}]{berezhkovskii_committors_2019}%
  \BibitemOpen
  \bibfield  {author} {\bibinfo {author} {\bibfnamefont {A.~M.}\ \bibnamefont
  {Berezhkovskii}}\ and\ \bibinfo {author} {\bibfnamefont {A.}~\bibnamefont
  {Szabo}},\ }\href {\doibase 10.1063/1.5079742} {\bibfield  {journal}
  {\bibinfo  {journal} {The Journal of Chemical Physics}\ }\textbf {\bibinfo
  {volume} {150}},\ \bibinfo {pages} {054106} (\bibinfo {year}
  {2019})}\BibitemShut {NoStop}%
\bibitem [{\citenamefont {Pande}\ \emph {et~al.}(2010)\citenamefont {Pande},
  \citenamefont {Beauchamp},\ and\ \citenamefont
  {Bowman}}]{pande2010everything}%
  \BibitemOpen
  \bibfield  {author} {\bibinfo {author} {\bibfnamefont {V.~S.}\ \bibnamefont
  {Pande}}, \bibinfo {author} {\bibfnamefont {K.}~\bibnamefont {Beauchamp}}, \
  and\ \bibinfo {author} {\bibfnamefont {G.~R.}\ \bibnamefont {Bowman}},\
  }\href {\doibase 10.1016/j.ymeth.2010.06.002} {\bibfield  {journal} {\bibinfo
   {journal} {Methods}\ }\textbf {\bibinfo {volume} {52}},\ \bibinfo {pages}
  {99} (\bibinfo {year} {2010})}\BibitemShut {NoStop}%
\bibitem [{\citenamefont {Chodera}\ and\ \citenamefont
  {Noé}(2014)}]{chodera2014markov}%
  \BibitemOpen
  \bibfield  {author} {\bibinfo {author} {\bibfnamefont {J.~D.}\ \bibnamefont
  {Chodera}}\ and\ \bibinfo {author} {\bibfnamefont {F.}~\bibnamefont {Noé}},\
  }\href {\doibase 10.1016/j.sbi.2014.04.002} {\bibfield  {journal} {\bibinfo
  {journal} {Current Opinion in Structural Biology}\ }\textbf {\bibinfo
  {volume} {25}},\ \bibinfo {pages} {135} (\bibinfo {year} {2014})}\BibitemShut
  {NoStop}%
\bibitem [{\citenamefont {Husic}\ and\ \citenamefont
  {Pande}(2018)}]{husic2018markov}%
  \BibitemOpen
  \bibfield  {author} {\bibinfo {author} {\bibfnamefont {B.~E.}\ \bibnamefont
  {Husic}}\ and\ \bibinfo {author} {\bibfnamefont {V.~S.}\ \bibnamefont
  {Pande}},\ }\href {\doibase 10.1021/jacs.7b12191} {\bibfield  {journal}
  {\bibinfo  {journal} {Journal of the American Chemical Society}\ }\textbf
  {\bibinfo {volume} {140}},\ \bibinfo {pages} {2386} (\bibinfo {year}
  {2018})}\BibitemShut {NoStop}%
\bibitem [{\citenamefont {Radhakrishnan}\ \emph {et~al.}(2024)\citenamefont
  {Radhakrishnan}, \citenamefont {Beaglehole}, \citenamefont {Pandit},\ and\
  \citenamefont {Belkin}}]{radhakrishnan2022feature}%
  \BibitemOpen
  \bibfield  {author} {\bibinfo {author} {\bibfnamefont {A.}~\bibnamefont
  {Radhakrishnan}}, \bibinfo {author} {\bibfnamefont {D.}~\bibnamefont
  {Beaglehole}}, \bibinfo {author} {\bibfnamefont {P.}~\bibnamefont {Pandit}},
  \ and\ \bibinfo {author} {\bibfnamefont {M.}~\bibnamefont {Belkin}},\ }\href
  {\doibase 10.1126/science.adi5639} {\bibfield  {journal} {\bibinfo  {journal}
  {Science}\ }\textbf {\bibinfo {volume} {383}},\ \bibinfo {pages} {1461}
  (\bibinfo {year} {2024})}\BibitemShut {NoStop}%
\bibitem [{\citenamefont {Chen}\ \emph
  {et~al.}(2023{\natexlab{a}})\citenamefont {Chen}, \citenamefont {Epperly},
  \citenamefont {Tropp},\ and\ \citenamefont {Webber}}]{chen_randomly_2023}%
  \BibitemOpen
  \bibfield  {author} {\bibinfo {author} {\bibfnamefont {Y.}~\bibnamefont
  {Chen}}, \bibinfo {author} {\bibfnamefont {E.~N.}\ \bibnamefont {Epperly}},
  \bibinfo {author} {\bibfnamefont {J.~A.}\ \bibnamefont {Tropp}}, \ and\
  \bibinfo {author} {\bibfnamefont {R.~J.}\ \bibnamefont {Webber}},\ }\href
  {\doibase 10.48550/arXiv.2207.06503} {\enquote {\bibinfo {title} {Randomly
  pivoted {Cholesky}: {Practical} approximation of a kernel matrix with few
  entry evaluations},}\ } (\bibinfo {year} {2023}{\natexlab{a}}),\ \bibinfo
  {note} {arXiv:2207.06503 [cs, math, stat]}\BibitemShut {NoStop}%
\bibitem [{\citenamefont {Ferguson}\ \emph {et~al.}(2011)\citenamefont
  {Ferguson}, \citenamefont {Panagiotopoulos}, \citenamefont {Kevrekidis},\
  and\ \citenamefont {Debenedetti}}]{ferguson2011nonlinear}%
  \BibitemOpen
  \bibfield  {author} {\bibinfo {author} {\bibfnamefont {A.~L.}\ \bibnamefont
  {Ferguson}}, \bibinfo {author} {\bibfnamefont {A.~Z.}\ \bibnamefont
  {Panagiotopoulos}}, \bibinfo {author} {\bibfnamefont {I.~G.}\ \bibnamefont
  {Kevrekidis}}, \ and\ \bibinfo {author} {\bibfnamefont {P.~G.}\ \bibnamefont
  {Debenedetti}},\ }\href {\doibase
  https://doi.org/10.1016/j.cplett.2011.04.066} {\bibfield  {journal} {\bibinfo
   {journal} {Chemical Physics Letters}\ }\textbf {\bibinfo {volume} {509}},\
  \bibinfo {pages} {1} (\bibinfo {year} {2011})}\BibitemShut {NoStop}%
\bibitem [{\citenamefont {Rohrdanz}\ \emph {et~al.}(2011)\citenamefont
  {Rohrdanz}, \citenamefont {Zheng}, \citenamefont {Maggioni},\ and\
  \citenamefont {Clementi}}]{rohrdanz2011determination}%
  \BibitemOpen
  \bibfield  {author} {\bibinfo {author} {\bibfnamefont {M.~A.}\ \bibnamefont
  {Rohrdanz}}, \bibinfo {author} {\bibfnamefont {W.}~\bibnamefont {Zheng}},
  \bibinfo {author} {\bibfnamefont {M.}~\bibnamefont {Maggioni}}, \ and\
  \bibinfo {author} {\bibfnamefont {C.}~\bibnamefont {Clementi}},\ }\href
  {\doibase 10.1063/1.3569857} {\bibfield  {journal} {\bibinfo  {journal} {The
  Journal of Chemical Physics}\ }\textbf {\bibinfo {volume} {134}},\ \bibinfo
  {pages} {124116} (\bibinfo {year} {2011})}\BibitemShut {NoStop}%
\bibitem [{\citenamefont {Griffiths}\ \emph {et~al.}(2023)\citenamefont
  {Griffiths}, \citenamefont {Klarner}, \citenamefont {Moss}, \citenamefont
  {Ravuri}, \citenamefont {Truong}, \citenamefont {Du}, \citenamefont
  {Stanton}, \citenamefont {Tom}, \citenamefont {Rankovi{\'c}}, \citenamefont
  {Jamasb}, \citenamefont {Deshwal}, \citenamefont {Schwartz}, \citenamefont
  {Tripp}, \citenamefont {Kell}, \citenamefont {Frieder}, \citenamefont
  {Bourached}, \citenamefont {Chan}, \citenamefont {Moss}, \citenamefont {Guo},
  \citenamefont {D{\"u}rholt}, \citenamefont {Chaurasia}, \citenamefont {Park},
  \citenamefont {Strieth-Kalthoff}, \citenamefont {Lee}, \citenamefont {Cheng},
  \citenamefont {Aspuru-Guzik}, \citenamefont {Schwaller},\ and\ \citenamefont
  {Tang}}]{griffiths2023gauche}%
  \BibitemOpen
  \bibfield  {author} {\bibinfo {author} {\bibfnamefont {R.-R.}\ \bibnamefont
  {Griffiths}}, \bibinfo {author} {\bibfnamefont {L.}~\bibnamefont {Klarner}},
  \bibinfo {author} {\bibfnamefont {H.}~\bibnamefont {Moss}}, \bibinfo {author}
  {\bibfnamefont {A.}~\bibnamefont {Ravuri}}, \bibinfo {author} {\bibfnamefont
  {S.~T.}\ \bibnamefont {Truong}}, \bibinfo {author} {\bibfnamefont
  {Y.}~\bibnamefont {Du}}, \bibinfo {author} {\bibfnamefont {S.~D.}\
  \bibnamefont {Stanton}}, \bibinfo {author} {\bibfnamefont {G.}~\bibnamefont
  {Tom}}, \bibinfo {author} {\bibfnamefont {B.}~\bibnamefont {Rankovi{\'c}}},
  \bibinfo {author} {\bibfnamefont {A.~R.}\ \bibnamefont {Jamasb}}, \bibinfo
  {author} {\bibfnamefont {A.}~\bibnamefont {Deshwal}}, \bibinfo {author}
  {\bibfnamefont {J.}~\bibnamefont {Schwartz}}, \bibinfo {author}
  {\bibfnamefont {A.}~\bibnamefont {Tripp}}, \bibinfo {author} {\bibfnamefont
  {G.}~\bibnamefont {Kell}}, \bibinfo {author} {\bibfnamefont {S.}~\bibnamefont
  {Frieder}}, \bibinfo {author} {\bibfnamefont {A.}~\bibnamefont {Bourached}},
  \bibinfo {author} {\bibfnamefont {A.~J.}\ \bibnamefont {Chan}}, \bibinfo
  {author} {\bibfnamefont {J.}~\bibnamefont {Moss}}, \bibinfo {author}
  {\bibfnamefont {C.}~\bibnamefont {Guo}}, \bibinfo {author} {\bibfnamefont
  {J.~P.}\ \bibnamefont {D{\"u}rholt}}, \bibinfo {author} {\bibfnamefont
  {S.}~\bibnamefont {Chaurasia}}, \bibinfo {author} {\bibfnamefont {J.~W.}\
  \bibnamefont {Park}}, \bibinfo {author} {\bibfnamefont {F.}~\bibnamefont
  {Strieth-Kalthoff}}, \bibinfo {author} {\bibfnamefont {A.}~\bibnamefont
  {Lee}}, \bibinfo {author} {\bibfnamefont {B.}~\bibnamefont {Cheng}}, \bibinfo
  {author} {\bibfnamefont {A.}~\bibnamefont {Aspuru-Guzik}}, \bibinfo {author}
  {\bibfnamefont {P.}~\bibnamefont {Schwaller}}, \ and\ \bibinfo {author}
  {\bibfnamefont {J.}~\bibnamefont {Tang}},\ }in\ \href
  {https://openreview.net/forum?id=vzrA6uqOis} {\emph {\bibinfo {booktitle}
  {Thirty-seventh Conference on Neural Information Processing Systems}}}\
  (\bibinfo {year} {2023})\BibitemShut {NoStop}%
\bibitem [{\citenamefont {Burn}\ and\ \citenamefont
  {Popelier}(2020)}]{burn2020creating}%
  \BibitemOpen
  \bibfield  {author} {\bibinfo {author} {\bibfnamefont {M.~J.}\ \bibnamefont
  {Burn}}\ and\ \bibinfo {author} {\bibfnamefont {P.~L.~A.}\ \bibnamefont
  {Popelier}},\ }\href {\doibase 10.1063/5.0017887} {\bibfield  {journal}
  {\bibinfo  {journal} {The Journal of Chemical Physics}\ }\textbf {\bibinfo
  {volume} {153}},\ \bibinfo {pages} {054111} (\bibinfo {year}
  {2020})}\BibitemShut {NoStop}%
\bibitem [{\citenamefont {Deringer}\ \emph {et~al.}(2021)\citenamefont
  {Deringer}, \citenamefont {Bartók}, \citenamefont {Bernstein}, \citenamefont
  {Wilkins}, \citenamefont {Ceriotti},\ and\ \citenamefont
  {Csányi}}]{deringer2021gaussian}%
  \BibitemOpen
  \bibfield  {author} {\bibinfo {author} {\bibfnamefont {V.~L.}\ \bibnamefont
  {Deringer}}, \bibinfo {author} {\bibfnamefont {A.~P.}\ \bibnamefont
  {Bartók}}, \bibinfo {author} {\bibfnamefont {N.}~\bibnamefont {Bernstein}},
  \bibinfo {author} {\bibfnamefont {D.~M.}\ \bibnamefont {Wilkins}}, \bibinfo
  {author} {\bibfnamefont {M.}~\bibnamefont {Ceriotti}}, \ and\ \bibinfo
  {author} {\bibfnamefont {G.}~\bibnamefont {Csányi}},\ }\href {\doibase
  10.1021/acs.chemrev.1c00022} {\bibfield  {journal} {\bibinfo  {journal}
  {Chemical Reviews}\ }\textbf {\bibinfo {volume} {121}},\ \bibinfo {pages}
  {10073} (\bibinfo {year} {2021})}\BibitemShut {NoStop}%
\bibitem [{\citenamefont {Burn}\ and\ \citenamefont
  {Popelier}(2022)}]{burn2022ichor}%
  \BibitemOpen
  \bibfield  {author} {\bibinfo {author} {\bibfnamefont {M.~J.}\ \bibnamefont
  {Burn}}\ and\ \bibinfo {author} {\bibfnamefont {P.~L.~A.}\ \bibnamefont
  {Popelier}},\ }\href {\doibase 10.1039/D2MA00673A} {\bibfield  {journal}
  {\bibinfo  {journal} {Materials Advances}\ }\textbf {\bibinfo {volume} {3}},\
  \bibinfo {pages} {8729} (\bibinfo {year} {2022})}\BibitemShut {NoStop}%
\bibitem [{\citenamefont {Tom}\ \emph {et~al.}(2023)\citenamefont {Tom},
  \citenamefont {Hickman}, \citenamefont {Zinzuwadia}, \citenamefont
  {Mohajeri}, \citenamefont {Sanchez-Lengeling},\ and\ \citenamefont
  {Aspuru-Guzik}}]{tom2023calibration}%
  \BibitemOpen
  \bibfield  {author} {\bibinfo {author} {\bibfnamefont {G.}~\bibnamefont
  {Tom}}, \bibinfo {author} {\bibfnamefont {R.~J.}\ \bibnamefont {Hickman}},
  \bibinfo {author} {\bibfnamefont {A.}~\bibnamefont {Zinzuwadia}}, \bibinfo
  {author} {\bibfnamefont {A.}~\bibnamefont {Mohajeri}}, \bibinfo {author}
  {\bibfnamefont {B.}~\bibnamefont {Sanchez-Lengeling}}, \ and\ \bibinfo
  {author} {\bibfnamefont {A.}~\bibnamefont {Aspuru-Guzik}},\ }\href {\doibase
  10.1039/D2DD00146B} {\bibfield  {journal} {\bibinfo  {journal} {Digital
  Discovery}\ }\textbf {\bibinfo {volume} {2}},\ \bibinfo {pages} {759}
  (\bibinfo {year} {2023})}\BibitemShut {NoStop}%
\bibitem [{\citenamefont {Musil}\ \emph {et~al.}(2021)\citenamefont {Musil},
  \citenamefont {Grisafi}, \citenamefont {Bartók}, \citenamefont {Ortner},
  \citenamefont {Csányi},\ and\ \citenamefont {Ceriotti}}]{musil2021physics}%
  \BibitemOpen
  \bibfield  {author} {\bibinfo {author} {\bibfnamefont {F.}~\bibnamefont
  {Musil}}, \bibinfo {author} {\bibfnamefont {A.}~\bibnamefont {Grisafi}},
  \bibinfo {author} {\bibfnamefont {A.~P.}\ \bibnamefont {Bartók}}, \bibinfo
  {author} {\bibfnamefont {C.}~\bibnamefont {Ortner}}, \bibinfo {author}
  {\bibfnamefont {G.}~\bibnamefont {Csányi}}, \ and\ \bibinfo {author}
  {\bibfnamefont {M.}~\bibnamefont {Ceriotti}},\ }\href {\doibase
  10.1021/acs.chemrev.1c00021} {\bibfield  {journal} {\bibinfo  {journal}
  {Chemical Reviews}\ }\textbf {\bibinfo {volume} {121}},\ \bibinfo {pages}
  {9759} (\bibinfo {year} {2021})}\BibitemShut {NoStop}%
\bibitem [{\citenamefont {Coifman}\ \emph {et~al.}(2008)\citenamefont
  {Coifman}, \citenamefont {Kevrekidis}, \citenamefont {Lafon}, \citenamefont
  {Maggioni},\ and\ \citenamefont {Nadler}}]{coifman_diffusion_2008}%
  \BibitemOpen
  \bibfield  {author} {\bibinfo {author} {\bibfnamefont {R.~R.}\ \bibnamefont
  {Coifman}}, \bibinfo {author} {\bibfnamefont {I.~G.}\ \bibnamefont
  {Kevrekidis}}, \bibinfo {author} {\bibfnamefont {S.}~\bibnamefont {Lafon}},
  \bibinfo {author} {\bibfnamefont {M.}~\bibnamefont {Maggioni}}, \ and\
  \bibinfo {author} {\bibfnamefont {B.}~\bibnamefont {Nadler}},\ }\href
  {\doibase 10.1137/070696325} {\bibfield  {journal} {\bibinfo  {journal}
  {Multiscale Modeling \& Simulation}\ }\textbf {\bibinfo {volume} {7}},\
  \bibinfo {pages} {842} (\bibinfo {year} {2008})}\BibitemShut {NoStop}%
\bibitem [{\citenamefont {No\'{e}}\ and\ \citenamefont
  {N\"{u}ske}(2013)}]{noe2013variational}%
  \BibitemOpen
  \bibfield  {author} {\bibinfo {author} {\bibfnamefont {F.}~\bibnamefont
  {No\'{e}}}\ and\ \bibinfo {author} {\bibfnamefont {F.}~\bibnamefont
  {N\"{u}ske}},\ }\href {\doibase 10.1137/110858616} {\bibfield  {journal}
  {\bibinfo  {journal} {Multiscale Modeling \& Simulation}\ }\textbf {\bibinfo
  {volume} {11}},\ \bibinfo {pages} {635} (\bibinfo {year} {2013})}\BibitemShut
  {NoStop}%
\bibitem [{\citenamefont {Chen}\ \emph
  {et~al.}(2023{\natexlab{b}})\citenamefont {Chen}, \citenamefont {Roux},\ and\
  \citenamefont {Chipot}}]{chen2023discovering}%
  \BibitemOpen
  \bibfield  {author} {\bibinfo {author} {\bibfnamefont {H.}~\bibnamefont
  {Chen}}, \bibinfo {author} {\bibfnamefont {B.}~\bibnamefont {Roux}}, \ and\
  \bibinfo {author} {\bibfnamefont {C.}~\bibnamefont {Chipot}},\ }\href
  {\doibase 10.1021/acs.jctc.3c00028} {\bibfield  {journal} {\bibinfo
  {journal} {Journal of Chemical Theory and Computation}\ }\textbf {\bibinfo
  {volume} {19}},\ \bibinfo {pages} {4414} (\bibinfo {year}
  {2023}{\natexlab{b}})}\BibitemShut {NoStop}%
\bibitem [{\citenamefont {Chen}\ and\ \citenamefont
  {Chipot}(2023)}]{chen2023chasing}%
  \BibitemOpen
  \bibfield  {author} {\bibinfo {author} {\bibfnamefont {H.}~\bibnamefont
  {Chen}}\ and\ \bibinfo {author} {\bibfnamefont {C.}~\bibnamefont {Chipot}},\
  }\href {\doibase 10.1017/qrd.2022.23} {\bibfield  {journal} {\bibinfo
  {journal} {QRB Discovery}\ }\textbf {\bibinfo {volume} {4}},\ \bibinfo
  {pages} {e2} (\bibinfo {year} {2023})}\BibitemShut {NoStop}%
\bibitem [{\citenamefont {Schütte}\ \emph {et~al.}(2023)\citenamefont
  {Schütte}, \citenamefont {Klus},\ and\ \citenamefont
  {Hartmann}}]{schutte2023overcoming}%
  \BibitemOpen
  \bibfield  {author} {\bibinfo {author} {\bibfnamefont {C.}~\bibnamefont
  {Schütte}}, \bibinfo {author} {\bibfnamefont {S.}~\bibnamefont {Klus}}, \
  and\ \bibinfo {author} {\bibfnamefont {C.}~\bibnamefont {Hartmann}},\ }\href
  {\doibase 10.1017/S0962492923000016} {\bibfield  {journal} {\bibinfo
  {journal} {Acta Numerica}\ }\textbf {\bibinfo {volume} {32}},\ \bibinfo
  {pages} {517–673} (\bibinfo {year} {2023})}\BibitemShut {NoStop}%
\bibitem [{\citenamefont {Kamada}\ and\ \citenamefont
  {Abe}(2006)}]{kamada2006support}%
  \BibitemOpen
  \bibfield  {author} {\bibinfo {author} {\bibfnamefont {Y.}~\bibnamefont
  {Kamada}}\ and\ \bibinfo {author} {\bibfnamefont {S.}~\bibnamefont {Abe}},\
  }in\ \href {https://link.springer.com/chapter/10.1007/11829898_13} {\emph
  {\bibinfo {booktitle} {IAPR Workshop on Artificial Neural Networks in Pattern
  Recognition}}}\ (\bibinfo {year} {2006})\BibitemShut {NoStop}%
\bibitem [{\citenamefont {Camps-Valls}\ \emph {et~al.}(2007)\citenamefont
  {Camps-Valls}, \citenamefont {Rodrigo-Gonzalez}, \citenamefont {Munoz-Mari},
  \citenamefont {Gomez-Chova},\ and\ \citenamefont
  {Calpe-Maravilla}}]{campsvalls2007hyperspectral}%
  \BibitemOpen
  \bibfield  {author} {\bibinfo {author} {\bibfnamefont {G.}~\bibnamefont
  {Camps-Valls}}, \bibinfo {author} {\bibfnamefont {A.}~\bibnamefont
  {Rodrigo-Gonzalez}}, \bibinfo {author} {\bibfnamefont {J.}~\bibnamefont
  {Munoz-Mari}}, \bibinfo {author} {\bibfnamefont {L.}~\bibnamefont
  {Gomez-Chova}}, \ and\ \bibinfo {author} {\bibfnamefont {J.}~\bibnamefont
  {Calpe-Maravilla}},\ }in\ \href {\doibase 10.1109/IGARSS.2007.4423671} {\emph
  {\bibinfo {booktitle} {IEEE International Geoscience and Remote Sensing
  Symposium}}}\ (\bibinfo {year} {2007})\BibitemShut {NoStop}%
\bibitem [{\citenamefont {Wu}\ \emph {et~al.}(2010)\citenamefont {Wu},
  \citenamefont {Guinney}, \citenamefont {Maggioni},\ and\ \citenamefont
  {Mukherjee}}]{wu_learning_2010}%
  \BibitemOpen
  \bibfield  {author} {\bibinfo {author} {\bibfnamefont {Q.}~\bibnamefont
  {Wu}}, \bibinfo {author} {\bibfnamefont {J.}~\bibnamefont {Guinney}},
  \bibinfo {author} {\bibfnamefont {M.}~\bibnamefont {Maggioni}}, \ and\
  \bibinfo {author} {\bibfnamefont {S.}~\bibnamefont {Mukherjee}},\ }\href@noop
  {} {\bibfield  {journal} {\bibinfo  {journal} {Journal of Machine Learning
  Research}\ }\textbf {\bibinfo {volume} {11}},\ \bibinfo {pages} {2175}
  (\bibinfo {year} {2010})}\BibitemShut {NoStop}%
\bibitem [{\citenamefont {Trivedi}\ \emph {et~al.}(2014)\citenamefont
  {Trivedi}, \citenamefont {Wang}, \citenamefont {Kpotufe},\ and\ \citenamefont
  {Shakhnarovich}}]{trivedi2014consistent}%
  \BibitemOpen
  \bibfield  {author} {\bibinfo {author} {\bibfnamefont {S.}~\bibnamefont
  {Trivedi}}, \bibinfo {author} {\bibfnamefont {J.}~\bibnamefont {Wang}},
  \bibinfo {author} {\bibfnamefont {S.}~\bibnamefont {Kpotufe}}, \ and\
  \bibinfo {author} {\bibfnamefont {G.}~\bibnamefont {Shakhnarovich}},\ }in\
  \href {https://dl.acm.org/doi/10.5555/3020751.3020836} {\emph {\bibinfo
  {booktitle} {Proceedings of the Thirtieth Conference on Uncertainty in
  Artificial Intelligence}}}\ (\bibinfo {year} {2014})\BibitemShut {NoStop}%
\bibitem [{\citenamefont {Beaglehole}\ \emph {et~al.}(2023)\citenamefont
  {Beaglehole}, \citenamefont {Radhakrishnan}, \citenamefont {Pandit},\ and\
  \citenamefont {Belkin}}]{beaglehole2023mechanism}%
  \BibitemOpen
  \bibfield  {author} {\bibinfo {author} {\bibfnamefont {D.}~\bibnamefont
  {Beaglehole}}, \bibinfo {author} {\bibfnamefont {A.}~\bibnamefont
  {Radhakrishnan}}, \bibinfo {author} {\bibfnamefont {P.}~\bibnamefont
  {Pandit}}, \ and\ \bibinfo {author} {\bibfnamefont {M.}~\bibnamefont
  {Belkin}},\ }\href@noop {} {\enquote {\bibinfo {title} {Mechanism of feature
  learning in convolutional neural networks},}\ } (\bibinfo {year} {2023}),\
  \Eprint {http://arxiv.org/abs/2309.00570} {arXiv:2309.00570 [stat.ML]}
  \BibitemShut {NoStop}%
\bibitem [{\citenamefont {Beaglehole}\ \emph {et~al.}(2024)\citenamefont
  {Beaglehole}, \citenamefont {Súkeník}, \citenamefont {Mondelli},\ and\
  \citenamefont {Belkin}}]{beaglehole2024average}%
  \BibitemOpen
  \bibfield  {author} {\bibinfo {author} {\bibfnamefont {D.}~\bibnamefont
  {Beaglehole}}, \bibinfo {author} {\bibfnamefont {P.}~\bibnamefont
  {Súkeník}}, \bibinfo {author} {\bibfnamefont {M.}~\bibnamefont {Mondelli}},
  \ and\ \bibinfo {author} {\bibfnamefont {M.}~\bibnamefont {Belkin}},\
  }\href@noop {} {\enquote {\bibinfo {title} {Average gradient outer product as
  a mechanism for deep neural collapse},}\ } (\bibinfo {year} {2024}),\ \Eprint
  {http://arxiv.org/abs/2402.13728} {arXiv:2402.13728 [cs.LG]} \BibitemShut
  {NoStop}%
\bibitem [{\citenamefont {Unke}\ \emph {et~al.}(2021)\citenamefont {Unke},
  \citenamefont {Chmiela}, \citenamefont {Sauceda}, \citenamefont {Gastegger},
  \citenamefont {Poltavsky}, \citenamefont {Schütt}, \citenamefont
  {Tkatchenko},\ and\ \citenamefont {Müller}}]{unke2021machine}%
  \BibitemOpen
  \bibfield  {author} {\bibinfo {author} {\bibfnamefont {O.~T.}\ \bibnamefont
  {Unke}}, \bibinfo {author} {\bibfnamefont {S.}~\bibnamefont {Chmiela}},
  \bibinfo {author} {\bibfnamefont {H.~E.}\ \bibnamefont {Sauceda}}, \bibinfo
  {author} {\bibfnamefont {M.}~\bibnamefont {Gastegger}}, \bibinfo {author}
  {\bibfnamefont {I.}~\bibnamefont {Poltavsky}}, \bibinfo {author}
  {\bibfnamefont {K.~T.}\ \bibnamefont {Schütt}}, \bibinfo {author}
  {\bibfnamefont {A.}~\bibnamefont {Tkatchenko}}, \ and\ \bibinfo {author}
  {\bibfnamefont {K.-R.}\ \bibnamefont {Müller}},\ }\href {\doibase
  10.1021/acs.chemrev.0c01111} {\bibfield  {journal} {\bibinfo  {journal}
  {Chemical Reviews}\ }\textbf {\bibinfo {volume} {121}},\ \bibinfo {pages}
  {10142} (\bibinfo {year} {2021})}\BibitemShut {NoStop}%
\bibitem [{\citenamefont {Williams}\ and\ \citenamefont
  {Seeger}(2000)}]{williams2000using}%
  \BibitemOpen
  \bibfield  {author} {\bibinfo {author} {\bibfnamefont {C.~K.~I.}\
  \bibnamefont {Williams}}\ and\ \bibinfo {author} {\bibfnamefont
  {M.}~\bibnamefont {Seeger}},\ }in\ \href
  {https://dl.acm.org/doi/10.5555/3008751.3008847} {\emph {\bibinfo {booktitle}
  {Proceedings of the 13th International Conference on Neural Information
  Processing Systems}}}\ (\bibinfo {year} {2000})\BibitemShut {NoStop}%
\bibitem [{\citenamefont {Rudi}\ \emph {et~al.}(2017)\citenamefont {Rudi},
  \citenamefont {Carratino},\ and\ \citenamefont {Rosasco}}]{rudi2017falkon}%
  \BibitemOpen
  \bibfield  {author} {\bibinfo {author} {\bibfnamefont {A.}~\bibnamefont
  {Rudi}}, \bibinfo {author} {\bibfnamefont {L.}~\bibnamefont {Carratino}}, \
  and\ \bibinfo {author} {\bibfnamefont {L.}~\bibnamefont {Rosasco}},\ }in\
  \href {https://dl.acm.org/doi/10.5555/3294996.3295145} {\emph {\bibinfo
  {booktitle} {Proceedings of the 31st International Conference on Neural
  Information Processing Systems}}}\ (\bibinfo {year} {2017})\BibitemShut
  {NoStop}%
\bibitem [{\citenamefont {Díaz}\ \emph {et~al.}(2023)\citenamefont {Díaz},
  \citenamefont {Epperly}, \citenamefont {Frangella}, \citenamefont {Tropp},\
  and\ \citenamefont {Webber}}]{díaz2023robust}%
  \BibitemOpen
  \bibfield  {author} {\bibinfo {author} {\bibfnamefont {M.}~\bibnamefont
  {Díaz}}, \bibinfo {author} {\bibfnamefont {E.~N.}\ \bibnamefont {Epperly}},
  \bibinfo {author} {\bibfnamefont {Z.}~\bibnamefont {Frangella}}, \bibinfo
  {author} {\bibfnamefont {J.~A.}\ \bibnamefont {Tropp}}, \ and\ \bibinfo
  {author} {\bibfnamefont {R.~J.}\ \bibnamefont {Webber}},\ }\href@noop {}
  {\enquote {\bibinfo {title} {Robust, randomized preconditioning for kernel
  ridge regression},}\ } (\bibinfo {year} {2023}),\ \Eprint
  {http://arxiv.org/abs/2304.12465} {arXiv:2304.12465 [math.NA]} \BibitemShut
  {NoStop}%
\bibitem [{\citenamefont {Doyle}\ and\ \citenamefont
  {Snell}(1984)}]{doyle1984random}%
  \BibitemOpen
  \bibfield  {author} {\bibinfo {author} {\bibfnamefont {P.}~\bibnamefont
  {Doyle}}\ and\ \bibinfo {author} {\bibfnamefont {J.}~\bibnamefont {Snell}},\
  }\href {\doibase 10.5948/upo9781614440222} {\emph {\bibinfo {title} {Random
  Walks and Electric Networks}}}\ (\bibinfo  {publisher} {American Mathematical
  Society},\ \bibinfo {year} {1984})\BibitemShut {NoStop}%
\bibitem [{\citenamefont {Hersh}\ and\ \citenamefont
  {Griego}(1969)}]{hersch1969brownian}%
  \BibitemOpen
  \bibfield  {author} {\bibinfo {author} {\bibfnamefont {R.}~\bibnamefont
  {Hersh}}\ and\ \bibinfo {author} {\bibfnamefont {R.~J.}\ \bibnamefont
  {Griego}},\ }\href {http://www.jstor.org/stable/24926310} {\bibfield
  {journal} {\bibinfo  {journal} {Scientific American}\ }\textbf {\bibinfo
  {volume} {220}},\ \bibinfo {pages} {66} (\bibinfo {year} {1969})}\BibitemShut
  {NoStop}%
\bibitem [{\citenamefont {Schmuland}(1999)}]{schmuland1999dirichlet}%
  \BibitemOpen
  \bibfield  {author} {\bibinfo {author} {\bibfnamefont {B.}~\bibnamefont
  {Schmuland}},\ }\href {http://www.jstor.org/stable/3316125} {\bibfield
  {journal} {\bibinfo  {journal} {The Canadian Journal of Statistics / La Revue
  Canadienne de Statistique}\ }\textbf {\bibinfo {volume} {27}},\ \bibinfo
  {pages} {683} (\bibinfo {year} {1999})}\BibitemShut {NoStop}%
\bibitem [{\citenamefont {Fukushima}\ \emph {et~al.}(2011)\citenamefont
  {Fukushima}, \citenamefont {Oshima},\ and\ \citenamefont
  {Takeda}}]{fukushima_dirichlet_2011}%
  \BibitemOpen
  \bibfield  {author} {\bibinfo {author} {\bibfnamefont {M.}~\bibnamefont
  {Fukushima}}, \bibinfo {author} {\bibfnamefont {Y.}~\bibnamefont {Oshima}}, \
  and\ \bibinfo {author} {\bibfnamefont {M.}~\bibnamefont {Takeda}},\
  }\href@noop {} {\emph {\bibinfo {title} {Dirichlet forms and symmetric
  {Markov} processes}}},\ \bibinfo {edition} {2nd}\ ed.\ (\bibinfo  {publisher}
  {De Gruyter},\ \bibinfo {address} {New York},\ \bibinfo {year}
  {2011})\BibitemShut {NoStop}%
\bibitem [{\citenamefont {Kanagawa}\ \emph {et~al.}(2018)\citenamefont
  {Kanagawa}, \citenamefont {Hennig}, \citenamefont {Sejdinovic},\ and\
  \citenamefont {Sriperumbudur}}]{kanagawa2018gaussian}%
  \BibitemOpen
  \bibfield  {author} {\bibinfo {author} {\bibfnamefont {M.}~\bibnamefont
  {Kanagawa}}, \bibinfo {author} {\bibfnamefont {P.}~\bibnamefont {Hennig}},
  \bibinfo {author} {\bibfnamefont {D.}~\bibnamefont {Sejdinovic}}, \ and\
  \bibinfo {author} {\bibfnamefont {B.~K.}\ \bibnamefont {Sriperumbudur}},\
  }\href@noop {} {\enquote {\bibinfo {title} {Gaussian processes and kernel
  methods: {A} review on connections and equivalences},}\ } (\bibinfo {year}
  {2018}),\ \Eprint {http://arxiv.org/abs/1807.02582} {arXiv:1807.02582
  [stat.ML]} \BibitemShut {NoStop}%
\bibitem [{Note1()}]{Note1}%
  \BibitemOpen
  \bibinfo {note}
  {Https://github.com/davidaristoff/Fast-Committor-Machine/}\BibitemShut
  {NoStop}%
\bibitem [{\citenamefont {Belkin}\ \emph {et~al.}(2019)\citenamefont {Belkin},
  \citenamefont {Hsu}, \citenamefont {Ma},\ and\ \citenamefont
  {Mandal}}]{belkin2019reconciling}%
  \BibitemOpen
  \bibfield  {author} {\bibinfo {author} {\bibfnamefont {M.}~\bibnamefont
  {Belkin}}, \bibinfo {author} {\bibfnamefont {D.}~\bibnamefont {Hsu}},
  \bibinfo {author} {\bibfnamefont {S.}~\bibnamefont {Ma}}, \ and\ \bibinfo
  {author} {\bibfnamefont {S.}~\bibnamefont {Mandal}},\ }\href {\doibase
  10.1073/pnas.1903070116} {\bibfield  {journal} {\bibinfo  {journal}
  {Proceedings of the National Academy of Sciences}\ }\textbf {\bibinfo
  {volume} {116}},\ \bibinfo {pages} {15849} (\bibinfo {year}
  {2019})}\BibitemShut {NoStop}%
\bibitem [{\citenamefont {Paszke}\ \emph {et~al.}(2019)\citenamefont {Paszke},
  \citenamefont {Gross}, \citenamefont {Massa}, \citenamefont {Lerer},
  \citenamefont {Bradbury}, \citenamefont {Chanan}, \citenamefont {Killeen},
  \citenamefont {Lin}, \citenamefont {Gimelshein}, \citenamefont {Antiga},
  \citenamefont {Desmaison}, \citenamefont {K\"{o}pf}, \citenamefont {Yang},
  \citenamefont {DeVito}, \citenamefont {Raison}, \citenamefont {Tejani},
  \citenamefont {Chilamkurthy}, \citenamefont {Steiner}, \citenamefont {Fang},
  \citenamefont {Bai},\ and\ \citenamefont {Chintala}}]{paszke2019pytorch}%
  \BibitemOpen
  \bibfield  {author} {\bibinfo {author} {\bibfnamefont {A.}~\bibnamefont
  {Paszke}}, \bibinfo {author} {\bibfnamefont {S.}~\bibnamefont {Gross}},
  \bibinfo {author} {\bibfnamefont {F.}~\bibnamefont {Massa}}, \bibinfo
  {author} {\bibfnamefont {A.}~\bibnamefont {Lerer}}, \bibinfo {author}
  {\bibfnamefont {J.}~\bibnamefont {Bradbury}}, \bibinfo {author}
  {\bibfnamefont {G.}~\bibnamefont {Chanan}}, \bibinfo {author} {\bibfnamefont
  {T.}~\bibnamefont {Killeen}}, \bibinfo {author} {\bibfnamefont
  {Z.}~\bibnamefont {Lin}}, \bibinfo {author} {\bibfnamefont {N.}~\bibnamefont
  {Gimelshein}}, \bibinfo {author} {\bibfnamefont {L.}~\bibnamefont {Antiga}},
  \bibinfo {author} {\bibfnamefont {A.}~\bibnamefont {Desmaison}}, \bibinfo
  {author} {\bibfnamefont {A.}~\bibnamefont {K\"{o}pf}}, \bibinfo {author}
  {\bibfnamefont {E.}~\bibnamefont {Yang}}, \bibinfo {author} {\bibfnamefont
  {Z.}~\bibnamefont {DeVito}}, \bibinfo {author} {\bibfnamefont
  {M.}~\bibnamefont {Raison}}, \bibinfo {author} {\bibfnamefont
  {A.}~\bibnamefont {Tejani}}, \bibinfo {author} {\bibfnamefont
  {S.}~\bibnamefont {Chilamkurthy}}, \bibinfo {author} {\bibfnamefont
  {B.}~\bibnamefont {Steiner}}, \bibinfo {author} {\bibfnamefont
  {L.}~\bibnamefont {Fang}}, \bibinfo {author} {\bibfnamefont {J.}~\bibnamefont
  {Bai}}, \ and\ \bibinfo {author} {\bibfnamefont {S.}~\bibnamefont
  {Chintala}},\ }in\ \href {https://dl.acm.org/doi/10.5555/3454287.3455008}
  {\emph {\bibinfo {booktitle} {Proceedings of the 33rd International
  Conference on Neural Information Processing Systems}}}\ (\bibinfo {year}
  {2019})\BibitemShut {NoStop}%
\bibitem [{\citenamefont {Loshchilov}\ and\ \citenamefont
  {Hutter}(2019)}]{loshchilov2019decoupled}%
  \BibitemOpen
  \bibfield  {author} {\bibinfo {author} {\bibfnamefont {I.}~\bibnamefont
  {Loshchilov}}\ and\ \bibinfo {author} {\bibfnamefont {F.}~\bibnamefont
  {Hutter}},\ }in\ \href {https://openreview.net/forum?id=Bkg6RiCqY7} {\emph
  {\bibinfo {booktitle} {Proceedings of the Seventh International Conference on
  Learning Representations}}}\ (\bibinfo {year} {2019})\BibitemShut {NoStop}%
\bibitem [{Note2()}]{Note2}%
  \BibitemOpen
  \bibinfo {note}
  {Https://www.plumed.org/doc-v2.7/user-doc/html/masterclass-21-4.html}\BibitemShut
  {NoStop}%
\bibitem [{Note3()}]{Note3}%
  \BibitemOpen
  \bibinfo {note}
  {Https://github.com/davidaristoff/Fast-Committor-Machine/}\BibitemShut
  {NoStop}%
\bibitem [{\citenamefont {Rotskoff}\ \emph {et~al.}(2022)\citenamefont
  {Rotskoff}, \citenamefont {Mitchell},\ and\ \citenamefont
  {Vanden-Eijnden}}]{rotskoff2022active}%
  \BibitemOpen
  \bibfield  {author} {\bibinfo {author} {\bibfnamefont {G.~M.}\ \bibnamefont
  {Rotskoff}}, \bibinfo {author} {\bibfnamefont {A.~R.}\ \bibnamefont
  {Mitchell}}, \ and\ \bibinfo {author} {\bibfnamefont {E.}~\bibnamefont
  {Vanden-Eijnden}},\ }in\ \href
  {https://proceedings.mlr.press/v145/rotskoff22a.html} {\emph {\bibinfo
  {booktitle} {Proceedings of the 2nd Mathematical and Scientific Machine
  Learning Conference}}}\ (\bibinfo {year} {2022})\BibitemShut {NoStop}%
\end{thebibliography}%

\end{document}